\documentclass[10pt]{amsart}
\usepackage{enumerate}
\usepackage{amsmath}
\usepackage{amssymb}
\usepackage[T1]{fontenc}
 \usepackage[utf8]{inputenc}
\usepackage{graphicx}
\usepackage{amsmath}
\usepackage{amssymb}
\usepackage{amsthm}
\usepackage[margin=1in]{geometry}

\usepackage{enumerate}

\newcommand{\om}{\omega}

\newtheorem{otherth}{\bf Theorem}

\newtheorem{otherl}{\bf Lemma}

\newtheorem{theorem}{Theorem}[section]
\newtheorem{lemma}[theorem]{Lemma}
\newtheorem{proposition}[theorem]{Proposition}

\theoremstyle{definition}

\newtheorem{definition}[theorem]{Definition}

\usepackage[dvipsnames]{xcolor}
\DeclareMathOperator*{\esssup}{ess\,sup}
\usepackage{graphicx}
\usepackage{amsfonts}
\usepackage{amsmath}
\usepackage{amssymb}
\usepackage{amsthm}
\usepackage{mathrsfs}
\usepackage[dvipsnames]{xcolor}
\usepackage{enumerate}
\usepackage{tikz}
\usepackage{enumerate}
\newcommand{\C}{\mathbb{C}}

\newcommand{\D}{\mathbb{D}}
\newcommand{\R}{\mathbb{R}}

\usepackage{hyperref}

\usepackage{url}

\usepackage{graphicx}

\hypersetup{
  colorlinks   = true,  
  urlcolor     = blue, 
  linkcolor    = black,  
  citecolor   = red 
}
\usepackage{verbatim}
\numberwithin{figure}{section}

\usepackage{amsthm}

\usepackage{cases}

\begin{document}

\title[Composition and Volterra type operators on large Bergman spaces]{Composition and Volterra type operators on large Bergman spaces with rapidly decreasing weights}

\author[M. Almoka]{Madhawi Almoka}
\address{Department of Mathematics, University of Hail, Saudi Arabia
\newline \indent Department of Mathematics and Statistics, University of Reading, England}
\email{m.almooka@uoh.edu.sa, m.m.a.almoka@pgr.reading.ac.uk}

\author[H. Arroussi]{Hicham Arroussi}
\address{Department of Mathematics and Statistics, University of Helsinki, Finland
\newline \indent Department of Mathematics and Statistics, University of Reading, England}

\email{arroussihicham@yahoo.fr, h.arroussi@reading.ac.uk}

\author[J. Virtanen]{Jani Virtanen}
\address{Department of Physics and Mathematics, University of Eastern Finland\newline \indent Department of Mathematics and Statistics, University of Reading, England}
\email{jani.virtanen@uef.fi, j.a.virtanen@reading.ac.uk}

\thanks{H. Arroussi was supported by the Horizon 2022 research and innovation programme of the European Union under the Marie Skłodowska-Curie Grant no.\,101109510. J. Virtanen was supported in part by Engineering and Physical Sciences Research Council (EPSRC) grant EP/Y008375/1.}

\begin{abstract}
We characterize boundedness, compactness and Schatten class properties of generalized Volterra-type integral operators acting between large Bergman spaces $A_\omega^p$ and $A_\omega^q$ for $0 <p, q\leq \infty$. 
To prove our characterizations, which involve Berezin-type integral transforms, we use the Littlewood-Paley formula of Constantin and Peláez and corresponding embedding theorems. Our results generalize the work on integration operators of Pau and Pel\'{a}ez in J. Funct. Anal. 259 (2010), 2727--2756.
\end{abstract}

\maketitle

\section{Introduction and main results }\label{Section 1}
Denote by $H(\D)$ the space of all analytic functions on the open unit disk $\D$ and by $dA$ the normalized area measure on $\mathbb{D}$. For $0< p < \infty$ and a positive function $\omega\in L^1(\D, dA)$,
the weighted Bergman space $A^p_{\omega}$ consists of those functions $f\in H(\D)$ for which
\[
\|f \|^p_{L^p_{\omega}} =\int_{\D}|f(z)|^p  \,\omega(z)^{p/2}\, dA(z) < \infty,
\]
and we set $A^{\infty}_\omega =  L^{\infty}(\omega^{1/2})\cap H(\mathbb{D})$, where
$\|f\|_{L^{\infty}(\omega^{1/2})} = \esssup_{z\in\D}\, |f(z)|\, \omega(z)^{1/2} < \infty.$

This paper is concerned with boundedness, compactness, and Schatten class membership of generalized Volterra-type integral operators, defined for analytic functions $\psi:\D\to\D$ and $g:\D\to\C$, by setting
\begin{equation}\label{operators}
   C_{(\psi,g)}f(z)= \int_{0}^{\psi(z)}f^\prime(\xi)\,g(\xi)\, d\xi 
\quad{\rm and}\quad
  C_{g}^{\psi}f(z)= \int_{0}^{\psi(z)} f(\xi)\,g(\xi)\, d\xi,  
\end{equation}
acting between $A^p_\omega$ and $A^q_\omega$ for $\omega$ in the class $\mathcal{W}$ that consists of the radial decreasing weights of the form $\omega(z)=e^{-2\varphi(z)}$, where 
$\varphi\in C^2(\D)$ is a radial function such that $\left(\Delta\varphi(z)\right)^{-1/2}
\asymp \tau(z)$ for some radial positive function $\tau(z) \in C^1(\D)$ that decreases to zero as $|z|\to 1^-$ and satisfies $\lim_{r\to 1^-} \tau'(r)=0$, and, in addition, we assume that there either exists a constant $C>0$ such that $\tau(r)(1-r)^{-C}$ increases for $r$ close to $1$ or if $\tau'(r)\log \frac{1}{\tau(r)}\rightarrow 0$ as $r\to 1^-$. The class $\mathcal{W}$ was introduced in \cite{BDK} in connection with sampling and interpolation. See also Section 7 of \cite{PP1} for several examples of weights in $\mathcal{W}$.

If $\psi(z)=z$, we denote the operators in \eqref{operators} by $I_g$ and $J_g$, respectively. Previously, Dostani\'c~\cite{D2} characterized boundedness and compactness of $J_g: A^2_{\om_\alpha}\to A^2_{\om_\alpha}$ with the prototypical weights $\omega_{\alpha}(z) = \exp(-b(1-|z|^2)^{-\alpha})$ in $\mathcal{W}$, where $b,\alpha>0$. Subsequently, Pau and Peláez \cite{PP1} extended Dostani\'c's results to all weights $w\in \mathcal{W}$ when $J_g$ acts from $A^p_\om$ to $A^q_\om$ for all $0<p,q\le \infty$. In the present work, we verify that the previous  characterizations agree with our results when $\psi(z)=z$. Further, we note that our results on $C_{\psi,g}$ are new even for $I_g$. For analogous results in the setting of standard Fock spaces $F^p_\alpha = H(\C)\cap L^p(\C, e^{-\alpha p |z|^2}dA)$ with $\alpha>0$, see the work of Mengestie~\cite{Tesfa, Tesfa00, Tesfa1}.

The operators $C_{\psi, g}$ and $C^g_\psi$ are closely related to the operators
\begin{equation}\label{operators2}
    GI_{(\psi,g)}f(z)=\int_{0}^{z} f^\prime(\psi(\xi))\,g(\xi)d\xi \quad {\rm and}\quad
    GV_{(\psi,g)}f(z)=\int_{0}^z f(\psi(\xi))\,g(\xi)d\xi,
\end{equation}
whose boundedness and compactness were recently studied in \cite{HGJ}. In addition to boundedness and compactness, we also characterize the Schatten class membership of $C^\psi_g$ and $GV_{\psi, g}$ while the case of the other two operators is currently out of our reach. 

Regarding terminology, the operators in $GI_{\psi, g}$ and $C_{\psi, g}$ are often called generalized Volterra companion operators because the particular choice $\psi(z) = z$ reduces them both to the Volterra companion operator $I_g$. They can also be thought of as generalized composition operators because the operators $GI_{\psi,g}$ and $C_{\psi,g}$ become composition operators $C_\psi$ (up to constants) when $g = \psi^\prime$ and $g = 1$, respectively.

\subsection{Main results}
When $0<p\le q<\infty$ or $0<q<p\le \infty$, our characterizations for boundedness and compactness of $C_{\psi, g}, C^\psi_g: A^p_\om \to A^q_\om$ involve the integral transform
\[M_{n,p,q}^{\psi} (g)(z) = \int_\mathbb{D} |k_{p,z} (\psi(\xi)|^q \,|g(\psi(\xi))|^q \,|\psi^\prime (\xi)|^q \,\frac{(1+\varphi^\prime (\psi(\xi))^{nq}}{(1+\varphi^\prime (\xi))^q} \omega(\xi)^{q/2}\,dA(\xi), \quad z\in \mathbb{D},\]
where $k_{p,z}= K_z/\|K_z\|_{A^p_{\om}}$ 
is defined via the reproducing kernel $K_z$ of $A^2_\om$. When $0<p\le q=\infty$, our characterizations are given in terms of
\begin{equation}\label{bdd 2nd case}
    N^{\psi, t}_{n,p, \infty}(g)(z)= \frac{|g(\psi(z))||\psi^\prime (z)|}{(1+\varphi^\prime (z))} (1+\varphi^\prime (\psi(z)))^n \frac{\omega(z)^\frac{1}{2}}{\omega(\psi(z))^\frac{1}{2}} \triangle\varphi(\psi(z))^t,
\end{equation}
where $t=\frac{1}{p}$ if $p<\infty$ and $t=0$ if $p=\infty$.

To state our main results, we write $B(X,Y)$ for bounded operators from $X$ to $Y$ and $K(X,Y)$ for compact operators.

\begin{theorem}\label{Theorem 1.1}
Let $\omega \in \mathcal{W}, \psi:\D\to \D$ be an analytic, and $g\in H(\D)$.

{\rm (A)} For $0<p\leq q<\infty$, 
$$
    C_{\psi,g}\in B(A^p_\om, A^q_\om) \iff M_{1,p,q}^\psi (g)\in L^\infty\  {\rm and}\ C_{\psi,g}\in K(A^p_\om, A^q_\om) \iff \lim_{|z|\to 1} M_{1,p,q}^\psi (g)(z) =0
$$
and
$$
    C^\psi_g\in B(A^p_\om, A^q_\om) \iff M_{0,p,q}^\psi (g)\in L^\infty\  {\rm and}\ C^\psi_g\in K(A^p_\om, A^q_\om) \iff \lim_{|z|\to 1} M_{0,p,q}^\psi (g)(z) =0.
$$

{\rm (B)} For $0<q<p\leq\infty$,
$$
    C_{\psi,g} \in B(A_\omega^p, A_\omega^q) \iff C_{\psi,g} \in K(A_\omega^p, A_\omega^q) \iff M_{1,p,q}^\psi (g)\in L^{s}(d\lambda)
$$
and
$$
    C^\psi_g \in B(A_\omega^p, A_\omega^q) \iff C^\psi_g \in K(A_\omega^p, A_\omega^q) \iff M_{0,p,q}^\psi (g)\in L^{s}(d\lambda)
$$
where $d\lambda(z)=dA(z)/\tau(z)^2$, $s=p/(p-q)$ if $p<\infty$, and $s=1$ if $p=\infty$.

{\rm (C)} For $0<p\le \infty$,
$$
    C_{\psi, g} \in B(A^p_\om, A^\infty_\om) \iff
    N^{\psi,t}_{1,p,\infty}(g) \in L^\infty\ {\rm and}\
    C_{\psi, g} \in K(A^p_\om, A^\infty_\om) \iff
    \lim_{|\psi(z)|\to 1} N^{\psi,t}_{1,p,\infty}(g)(z) \to 0
$$
and
$$
    C^\psi_g \in B(A^p_\om, A^\infty_\om) \iff
    N^{\psi,t}_{0,p,\infty}(g) \in L^\infty\ {\rm and}\
    C^\psi_g \in K(A^p_\om, A^\infty_\om) \iff
    \lim_{|\psi(z)|\to 1} N^{\psi,t}_{0,p,\infty}(g)(z) \to 0,
$$
where $t=\frac{1}{p}$ if $p<\infty$ and $t=0$ if $p=\infty$.
\end{theorem}

Our next main result determines when two operators $C^\psi_g$ and $GV_{\psi, g}$ belong to the Schatten $p$-class $S^p(A^2_\om)$ for every $0<p<\infty$.

\begin{theorem}\label{Theorem 1.2}
Let $0<p<\infty$, $\omega \in \mathcal{W}$, $\psi:\D\to \D$ be an analytic, and $g\in H(\D)$. Then
$$
    C^\psi_g\in S_p(A^2_\om) \iff
    M_{0,2,2}^{\psi}(g) \in L^{p/2}(d\lambda)
$$
and
$$ 
    GV_{\psi, g}\in S_p(A^2_\om)
    \iff 
z\mapsto \int_\mathbb{D} |k_{z}(\psi(\xi))|^2 \frac{|g(\xi)|^2\omega(\xi)}{(1+\varphi^\prime(\xi))^q} dA(\xi) \in L^{p/2}(d\lambda),
$$
where $d\lambda(z) = dA(z)/\tau(z)^2$.
\end{theorem}

\subsection{Outline} 
In Section \ref{Section 2}, we state various estimates for the reproducing kernel $K_z$, which play an important role in our work, recall useful test functions that were used in \cite{PP1} to treat the operators $J_g$ and geometric characterizations of Carleson measures, and also discuss embedding theorems and the basic theory of Schatten class operators.

Section \ref{Section 3} deals with boundedness and compactness. In particular, we prove Theorem \ref{Theorem 1.1}, provide simpler necessary conditions, and show that our characterizations for boundedness and compactness agree with those of Pau and Pel\`aez \cite{PP1}. Finally, in Section \ref{Section 4}, we prove Theorem \ref{Theorem 1.2} and again show that it agrees with the characterizations of Pau and Pel\`aez~\cite{PP1} for $J_g$ to be in $S_p(A^2_\om)$.

\section{Preliminaries}\label{Section 2}

Throughout our work, we first need to use several times generalizations of   Carleson measure for $A_\omega^p$ in  \cite{Hi-w} and the following Littlewood-Paley type formulas  \cite{Con&Pel}:
\begin{align}\label{eqnorm}
    \|f\|_{A_\omega ^p}^p &\asymp |f(0)|+\int_\mathbb{D} |f^\prime (z)|^p \frac{\omega(z)^{p/2}}{(1+\varphi^\prime(z))^p}dA(z)
\end{align}
\text{and,}
\begin{align}\label{eqnorm1}
    \|f\|_{A_\omega^\infty} &\asymp|f(0)|+\sup_{z\in\mathbb{D}} |f^\prime(z)| \frac{\omega(z)^{1/2}}{(1+\varphi^\prime(z))}.
\end{align}
Moreover, we need to consider the pullback measure
\begin{equation}\label{eq:pullback}
    \mu_{\psi, \omega,g} (E)=\int_{\psi^-1 (E)} |g(\psi(\xi))|^q\,|\psi^\prime (z)|^q\, \frac{\omega (z)^\frac{q}{2}}{(1+\varphi ^\prime (z))^q} dA(z),
\end{equation}
where $E$ is a Borel subset of $\mathbb{D}$, and $g$ is analytic function in $\mathbb{D}$.
By definition of $\mu_{\psi,\omega, g}$ on $\mathbb{D},$ for each $f\in A_\omega ^q$:
$$\int_\mathbb{D} |f^\prime(\psi(z))|^q|g(\psi(z))|^q|\psi^\prime (z)|^q \frac{\omega(z)^\frac{q}{2}}{(1+\varphi^\prime(z))^q} dA(z)=\int_\mathbb{D} |f^\prime (z)|^q d\mu_{\psi, \omega,g},$$ 
see \cite [Theorem C]{HALMOS}.
Furthermore, we define another measure 
$$d\nu_{\psi,\omega,g}=(1+\varphi^\prime (z))^q \omega(z)^\frac{-q}{2} d\mu_{\psi,\omega,g}, \quad z\in \mathbb{D}.$$

In what follows in this section, we define some further key concepts and recall previous results that are needed in our work.

\begin{definition}
A positive function $\tau$ on $\D$ is said to be of class $\mathcal{L}$ if  it satisfies the following two properties:
\begin{itemize}
\item[(A)] There is a constant $c_ 1$ such that \begin{equation}\label{def tau}
\tau(z)\le c_ 1\,(1-|z|)\,\qquad \text{for all}\,\qquad z\in \D;
\end{equation}
\item[(B)] There is a constant $c_ 2$ such that $|\tau(z)-\tau(\zeta)|\le c_ 2\,|z-\zeta|$ for all $z,\zeta\in \D$.
\end{itemize}
We also use the notation
 \[m_\tau : = \displaystyle \frac{\min(1, c_1^{-1}, c_2^{-1})}{4},\]
where $c_1$ and $c_2$ are the constants appearing in the previous definition.
\end{definition}
 For $a\in \D$ and $\delta>0$, we use  $D_\delta(a)$ to denote the Euclidean disc centered at $a$ and  having radius $\delta \tau(a)$.
It is easy to see from conditions (A) and (B) (see \cite[Lemma 2.1]{PP1}) that if $\tau \in\mathcal{L}$  and $ z \in  D_\delta(a) ,$ then

\begin{equation}\label{eqn:asymptau}
 \frac{1}{2}\,\tau(a)\leq \tau(z) \leq  2 \,\tau(a),
\end{equation}
for sufficiently small $\delta > 0 ,$ that is, for $\delta \in (0, m_\tau).$ This fact will be used many times in this work.
\begin{definition}\label{definit}
We say that a weight $\omega$ is of \emph{class  $\mathcal{L}^*$} if it is of the form $\omega =e^{-2\varphi}$, where $\varphi \in C^2(\D)$ with $\Delta \varphi>0$, and $\big (\Delta \varphi (z) \big )^{-1/2}\asymp \tau(z)$, with $\tau(z)$ being a function in the class $\mathcal{L}$. Here $\Delta$ denotes the classical Laplace operator.
\end{definition}
\begin{otherl}\label{lem:subHarmP}\cite{Hi-w}:
Let $\omega \in \mathcal{L}^*$, $0<p<\infty$, and  $z\in \D$. If $\beta \in \mathbb{R},$ there exists $ M \geq 1$ such that
\[ |f(z)|^p \omega(z)^{\beta} \leq  \frac{M}{\delta^2\tau(z)^2}\int_{D_\delta(z)}{|f(\xi)|^p \omega(\xi)^{\beta} \   dA(\xi)},\]
for all $ f \in H(\mathbb{D})$  and all sufficiently small $\delta > 0 .$ 
\end{otherl}

 Using  the preceding lemma and the fact that there exists $r_0\in [0,1)$ such that for all $a\in \D$ with $1>|a|>r_0,$ and any $\delta>0$ small enough we have 
$$
\varphi'(a)\asymp \varphi'(z), \quad z\in D_\delta(a)
$$
(see statement $(d)$  in \cite[Lemma 32]{Co-Pe}),
one has 
\begin{equation}\label{Eq-gamma}
|f(z)|^p \frac{\omega(z)^{\beta}}{(1+\varphi'(z))^{\gamma}} \lesssim  \frac{1}{\delta^2\tau(z)^2}\int_{D_\delta(z)}{|f(\xi)|^p \frac{\omega(\xi)^{\beta}}{(1+\varphi'(\xi))^{\gamma}}  dA(\xi)}, 
\end{equation}
 for $\beta, \gamma \in \mathbb{R}.$
 
The following lemma gives the upper estimates for the derivatives of functions
in $A^p_\omega.$  In fact, its proof is the same in the case
of a doubling measure $\Delta\varphi$ which can be found in lemma 19 of \cite{MMO}. For our
setting, see \cite{xiaof, O}.

\begin{otherl}\label{DERI}\cite{Hi-w}:
Let $ \omega\in\mathcal{L}^*$ and $0 < p < \infty.$  For any $\delta_0 > 0$ sufficiently small  there exists a constant $C(\delta_0)> 0$ such that $$|f'(z)|^{p}\omega(z)^{p/2} \leq \frac{C(\delta_0)}{\tau(z)^{2+2p}}\bigg(\int_{D(\delta_0\tau(z)/2)}{|f(\xi)|^p \,\omega(\xi)^{p/2} dA(\xi)}\bigg)^{1/p}, $$
for all $ f \in H(\mathbb{D})$.
\end{otherl}
The following lemma on coverings is due to
Oleinik, see \cite{O}.
\begin{otherl}\label{lem:Lcoverings}\cite{Hi-w}:
Let $\tau$ be a positive function on $\D$ of class
$\mathcal{L}$, and let $\delta\in (0,m_{\tau})$. Then there exists
a sequence of points $\{z_ n\}\subset \D$ such that the following
conditions are satisfied:
\begin{enumerate}
\item[$(i)$] \,$z_ n\notin D_\delta(z_ k)$, \,$n\neq k$.
\item[$(ii)$] \, $\bigcup_ n D_\delta(z_ n)=\D$.
\item[$(iii)$] \, $\tilde{D}_\delta(z_ n)\subset
D_{3\delta}(z_ n)$, where $\tilde{D}_\delta (z_
n)=\bigcup_{z\in D_\delta(z_ n)}D_\delta(z)$,
$n=1,2,\dots$
\item[$(iv)$] \,$\big \{D_{3\delta}(z_ n)\big \}$ is a
covering of $\D$ of finite multiplicity $N$.
\end{enumerate}
\end{otherl}
The multiplicity $N$ in the previous lemma is independent of $\delta$, and it is easy to see that one can take, for example, $N=256$. Any sequence satisfying the conditions in Lemma \ref{lem:Lcoverings} will be called a $(\delta,\tau)$-lattice. Note that $|z_n|\to 1^{-}$ as $n\to \infty.$ In what follows, the sequence $\{z_n\}$ will always refer to the sequence chosen in Lemma \ref{lem:Lcoverings}.

\subsection{Reproducing kernel estimates}\label{sec kernel}
Recall that $k_{p,z}$ is  the normalized reproducing kernel in $A^p_{\om},$ that is $$k_{p,z}= |K_z| / \|K_z\|_{A^p_{\om}},\, z\in \D.$$
 The next result (see \cite{BDK,LR1,PP1} for $(a)$ when $p=2$ and for every $p> 0$ see \cite{xiaof}.  The statement $(b)$ is an estimate of the reproducing kernel function for points close to the diagonal. Despite that this result is stated in \cite[Lemma 3.6]{LR2} we offer here a proof based on $(a),$ for $p=2$, since the conditions on the weights are slightly different.   
\begin{otherth}\label{RK-PE}\cite{Hi-w}:
Let $K_z $ be the reproducing kernel of  $A^2_{\omega}.$ Then
\begin{enumerate}
\item[(a)] For  $\om\in \mathcal{W}$ and $0< p< \infty$, one has
\begin{equation}\label{Eq-NE}
\|K_z\|_{A^p_{\omega}}   \asymp  \omega(z)^{-1/2}\, \tau(z)^{2(1-p)/p},\qquad z\in \D.
\end{equation}
\begin{equation}\label{Eq-INF}
\|K_z\|_{A^\infty_{\omega}}   \asymp  \omega(z)^{-1/2}\, \tau(z)^{-2},\qquad z\in \D.
\end{equation}
\item[(b)] For all sufficiently small $\delta \in (0,m_{\tau})$ and $\om \in \mathcal{W},$  one has
\begin{equation}\label{RK-Diag}
|K_ z(\zeta)| \asymp \|K_ z\|_{A^2_{\omega}}\cdot \|K_{\zeta}\|_{A^2_{\omega}} ,\qquad \zeta \in D_\delta(z).
\end{equation}
\end{enumerate}
\end{otherth}

The next lemma generalizes the statement $(a)$ of the above theorem.
\begin{otherl}\label{nEstim}\cite{Hi-w}:
Let $K_z $ be the reproducing kernel of  $A^2_{\omega}$ where $\omega$ is a weight in  the class $\mathcal{W}$. For each $z\in \D,$  $0< p < \infty$ and $\beta \in \R,$ one has
\begin{equation}\label{Eq-NE1}
  \int_{\D} |K_{z}(\xi)|^p\,\omega(\xi)^{p/2}\,\tau(\xi)^{\beta}\,dA(\xi) \leq C \om(z)^{-p/2}\, \tau(z)^{2(1-p)+\beta}.
\end{equation}
\end{otherl}

 \begin{otherl}\label{lem:RK-PE1}\cite{Hi-w}:
Let $K_z $ be the reproducing kernel of  $A^2_{\omega}$ where $\omega$ is a weight in  the class $\mathcal{W}$. Then
\begin{enumerate}
\item[(a)] For each $z\in \D,$ $0< p\le \infty,$ and  $0< q < \infty,$ one has
\begin{equation}\label{eqn:Eq-NE1}
|k_{p,z}(\zeta)|^q   \asymp \tau(z)^{2(1-\frac{q}{p})}|k_{q,z}(\zeta)|^q,\qquad \zeta\in \D.
\end{equation}
\item[(b)] For $q=\infty,$ one has $$|k_{p,z}(\zeta)|  \asymp \tau(z)^{-2/p}|k_{q,z}(\zeta)|,\qquad \zeta\in \D.$$
\item[(c)] For all $\delta \in (0,m_{\tau})$ sufficiently small, one has
\begin{equation}\label{eqn:RK-Diag1}
|k_ {p,z}(\zeta)|^p \, \omega(\zeta)^{p/2} \asymp \tau(z)^{-2} ,\qquad \zeta \in D_\delta(z).
\end{equation}
\end{enumerate}
\end{otherl}

\subsection{Test functions}\label{sec2:2}
It is known that having an appropriate family of test functions in a space of analytic functions X can help characterize the q-Carleson measures for X. In this
section we will do the job for the spaces $A^p_\omega.$
The following result on test functions was obtained in  \cite{PP1} and Lemma 3.3 in \cite{BDK} and we can refer also to Lemma C in \cite{Galano}. Without loss of generality, we modified the  original version by taking  $\omega(z)^{p/2}$ instead of $\omega(z),$ for the case $0<p<\infty.$ 
\begin{otherl}\label{Borichevlemma}\cite{PP2}:
Let $n\in\mathbb{N}\setminus\{0\}$ and $\omega\in \mathcal{W}$. There is a number $\rho_0 \in (0,1)$ such that for each $a\in \D$ with $|a|>\rho_0$
there is a function $F_{a,n}$ analytic in $\D$ with
\begin{equation}\label{BL1}
|F_{a,n}(z)|\,\omega(z)^{1/2}\asymp 1 \quad \textrm{ if }\quad
|z-a|<\tau(a),
\end{equation}
and
\begin{equation}\label{BL2}
|F_{a,n}(z)|\,\omega(z)^{1/2}\lesssim \min \left
(1,\frac{\min\big(\tau(a),\tau(z)\big)}{|z-a|} \right )^{3n}, \quad
z\in \D.
\end{equation}
Moreover,  
\begin{enumerate}
\item[(a)] For $0<p<\infty,$ the function $F_{a,n}$  belongs to $A^p(\omega)$  with $$\|F_{
a,n}\|_{A^p_\om}\asymp \tau(a)^{2/p}.$$
\item[(b)]  For $p=\infty,$ the function $F_{a,n}$  belongs to $A^\infty_{\omega}$  with $$\|F_{
a,n}\|_{A^\infty_{\omega}}\asymp 1.$$ 
\end{enumerate}
\end{otherl}
As a consequence we have the following  pointwise estimates for the derivative of the test functions $F_{a,n}.$
\begin{lemma}\label{BL3}
Let $n\in\mathbb{N}\setminus\{0\}$ and $\omega\in \mathcal{W}$.
For any $\delta>0$ small enough,
\begin{equation}
|F'_{a,n}(z)|\,\omega(z)^{1/2}\asymp 1+\varphi'(z), \quad z \in D_{\delta}(a).
\end{equation}
\end{lemma}
The next Proposition is some partial result about  the  atomic decomposition on $A^p_\om$ and its proof follows  easily from Lemma \ref{Borichevlemma}.
\begin{proposition}\label{At1-pp}\cite{PP2}:
Let $n\geq 2$  and $\omega \in \mathcal{W}.$ Let  $\{z_k\}_{k\in\mathbb{N}}\subset\D$ be the sequence  defined in Lemma \ref{lem:Lcoverings}.
\begin{enumerate}
\item[(a)] For  $0 < p < \infty,$ the function given by
$$ F(z): = \sum_{\substack{k = 0}}^{\infty}\lambda_k \,\,\frac{F_{z_k,n}(z)}{\tau(z_k)^{2/p}}$$ belongs to $A^p_\om$ for every sequence $\lambda=\lbrace\lambda_k\rbrace\in \ell^{p}$ . Moreover,$$\|F\|_{A^p_\om}\lesssim \|\lambda \|_{\ell^p}.$$
\item[(b)] For $p=\infty,$ the function given by
 $$ F(z): = \sum_{\substack{k = 0}}^{\infty}\lambda_k \,F_{z_k,n}(z)$$ belongs to $A^\infty_\om$ for every sequence $\lambda=\lbrace\lambda_k\rbrace\in \ell^{\infty}$ . Moreover,$$\|F\|_{A^\infty_{\omega}}\lesssim \|\lambda \|_{\ell^{\infty}}.$$
\end{enumerate}
\end{proposition}
\subsection{Geometric characterizations of Carleson measures}
Let $\mu$ be  a positive measure on $\D.$ Denote by $\widehat{\mu_{\delta}}$  the averaging function defined as
$$\widehat{\mu_{\delta}}(z)= \mu(D_{\delta}(z))\cdot\tau(z)^{-2}, \quad z\in \D,$$ 
and also a general Berezin transform of $\mu$  given by 
$$G_{t}(\mu)(z) = \int_{\D} |k_{t,z}(\zeta)|^t\, \om(\zeta)^{t/2}\, d\mu(\zeta), $$ for every $t>0$ and $z\in \D.$

In this section we recall recent characterizations of $q$-Carleson measures for $A^p_\omega$ for any $0 <p, q\leq\infty$ in terms of the averaging function $\widehat{\mu_{\delta}}$ and the general Berezin transform $G_t(\mu)$. For the proofs of all theorems in this section, see Section 3 of \cite{Hi-w}.
\subsubsection{Carleson measures }\label{sec3:1}
 We begin with the definition of $q$-Carleson measures.
\begin{definition}
Let $\mu$ be a positive measure on $\D$ and fix  $0< p,q< \infty.$ We say  that  $\mu$ is  a $q$-Carleson measure for $A^p_{\om}$  if the embdding operator $I_{\mu}:A^p_{\omega}\longrightarrow L^q_\omega$ is bounded. 
That is, $$\| I_\mu f\|_{L_\omega^q} \lesssim \|f\|_{A^p_\omega},$$ for $f\in A^p_{\omega}$ where $I_{\mu}$ is the identity and the expression $L_\omega^q$ mean $L_{\omega}^{q}(d\mu) :=L^q(\mathbb{D},\omega^{q/2}d\mu)$.
\end{definition}

The following theorem  characterizes the  $q$-Carleson measures when $0 < p\le q.$
\begin{otherth}\label{thm:CMPQ}
Let $\mu$  be a finite positive Borel measure on $\D.$ Assume $0<p \le q<\infty,$ $s=p/q,$ $1/s < t< \infty.$ The following conditions are  all equivalent:
\begin{enumerate}
\item[{\rm (a)}]  $\mu$ is a $q$-Carleson measure for $A^p_\omega;$
\item[{\rm (b)}]  The function  $$ \tau(z)^{2(1-1/s)}G_{t}(\mu)(z)$$ belongs to   $L^{\infty}(\D, dA).$ 
 \item[{\rm (c)}] The function  $$\tau(z)^{2(1-1/s)}\widehat{\mu_{\delta}}(z)$$ belongs to $L^{\infty}(\D, dA)$ for any small enough $\delta>0. $
\end{enumerate}
\end{otherth}
Now we charaterize $q$-Carleson measure for the case $0< q<p<\infty.$  

\begin{otherth}\label{thm:CMOP}
 Let $\mu$   be a finite positive Borel measure on $\D.$ Assume $0< q < p<\infty$  and $s=p/q.$  The following conditions are  all equivalent:
\begin{enumerate}
\item[{\rm (a)}]   $\mu$ is a $q$-Carleson measure for $A^p_\om$;
 \item[{\rm (b)}]   For any  (or some ) $r>0$, we have 
$$\widehat{\mu_{r}} \in L^{p/(p-q)}(\D, dA).$$ 
\item[{\rm (c)}] For any $t>1,$ $$ G_{t}(\mu)\in L^{p/(p-q)}(\D, dA).$$
\end{enumerate}
\end{otherth}
\subsubsection{Vanishing Carleson measures}\label{sec3:2}
\begin{definition} Let $\mu$ be a positive measure on $\D$ and fix $0< p,q<  \infty.$ We say  that  $\mu$ is  a vanishing $q$-Carleson measure for $A^p_\om$ if the inclusion $I_{\mu}:A^p_{\om}\longrightarrow L^q_\om$ is compact, or equivalently, if $$\int_{\D}|f_n(z)|^q\,\om(z)^{q/2}\, d\mu(z)\to 0,$$
whenever $f_n$ is bounded in $A^p_{\om}$  and  converges to zero uniformly on each compact subsets of $\D.$ 
\end{definition}
Next, we  characterize vanishing  $q$-Carleson measures for $A^p_\om$ whether  $0< p\le q < \infty$ or $0< q < p < \infty.$ We begin with the case  $0< p \le q<\infty.$
\begin{otherth}\label{thm:VCMPQ}
Given $\tau \in \mathcal{L^*},$  let $\mu$  be a finite positive Borel measure on $\D.$ Assume $0< p \le q<\infty,$ $s=p/q,$ $1/s < t< \infty.$ The following statements are  all equivalent:
\begin{enumerate}
\item[{\rm (a)}]   $\mu$ is a  vanishing $q$-Carleson measure for $A^p_\om.$
\item[{\rm (b)}]   $ \tau(z)^{2(1-1/s)}G_{t}(\mu)(z)\to 0 $ as $|z|\to 1^{-}.$
\item[{\rm (c)}]  $\tau(z)^{2(1-1/s)}\widehat{\mu_{\delta}}(z)\to 0$ as $|z|\to 1^{-},$ for any small enough $\delta>0.$ 
\end{enumerate}
\end{otherth}
The following theorem characterizes vanishing $q$-Carleson measures for $A^p_\om$  when $p=\infty$ and $0<  q <\infty$ in terms of the $t$-Berezin transform $G_{t}(\mu)$ and the averaging function $\widehat{\mu}_{\delta}.$ 
\begin{otherth}\label{thm:CMOP2}
Given $\tau \in \mathcal{L^*},$  let $\mu$  be a finite positive Borel measure on $\D.$ Assume $0< q < \infty.$  The following conditions are  all equivalent:
\begin{enumerate}
\item[{\rm (a)}]   $\mu$ is a   $q$-Carleson measure for $A^{\infty}_\om.$
\item[{\rm (b)}]   $\mu$ is a  vanishing $q$-Carleson measure for $A^{\infty}_\om.$ 
\item[{\rm (c)}] For any small enough $\delta>0$ , we have  $$\widehat{\mu_{\delta}} \in L^{1}(\D, dA).$$
\item[(d)] For any small enough $t>0$ , we have  $$G_t(\mu) \in L^{1}(\D, dA).$$
\end{enumerate}
\end{otherth}
\begin{otherth}\label{thm:VCMQP}
Given $\tau \in \mathcal{L^*},$  let $\mu$  be a finite positive Borel measure on $\D.$ Assume that  $0< q < p<\infty.$ The following statements are  equivalent:
\begin{enumerate}
\item[{\rm (a)}]   $\mu$ is a   $q$-Carleson measure for $A^p_\om.$ 
\item[{\rm (b)}]   $\mu$ is a  vanishing $q$-Carleson measure for $A^p_\om$ .
\end{enumerate}
\end{otherth}
 
\subsection{Embedding theorems}
The embedding theorems of $S^p_\omega$ into $L^q(\D,d\mu),$ for  $0<p,q\le \infty$ and $\omega \in \mathcal{W},$
where $$S^p_\omega:=\Bigg\{f\in H(\D): \int_{\D}|f(z)|^p\,\frac{\omega(z)^{p/2}}{(1+\varphi'(z))^p}dA(z)\,<\infty \Bigg\}$$
and 
$$S^\infty_\omega:=\Bigg\{f\in H(\D): \sup_{z\in\D}|f(z)|\,\frac{\omega(z)^{1/2}}{(1+\varphi'(z))}\,<\,\infty \Bigg\}.$$
For the proofs of all theorems in this section, see Section 4 of \cite{HGJ}. We start with the case $0<p\le q< \infty.$
\begin{otherl}\label{embed1}
    Let $\omega \in \mathcal{W}$ and $0<p\leq q<\infty$. Let $\mu$ be a finite positive Borel measure on $\mathbb{D}$. Then
    \begin{enumerate}
        \item[(i)] $I_\mu : S_\omega ^p \longrightarrow L^q (\mathbb{D} ,d\mu)$ is bounded if and only if for each $\delta >0$ small enough,
        $$K_{\mu,\omega} = \sup_{z\in\mathbb{D}} \frac{1}{\tau(z)^{2q/p}} \int_{D_\delta(z)}(1+\varphi^\prime (\xi))^q \omega(\xi)^{-q/2} d\mu(\xi)<\infty.$$
        \item[(ii)]$I_\mu : S_\omega ^p \longrightarrow L^q (\mathbb{D} ,d\mu)$ is compact if and only if
        $$\lim_{|z|\rightarrow 1^-} \int_{D_\delta(z)}(1+\varphi^\prime (\xi))^q \omega(\xi)^{-q/2} d\mu(\xi)=0.$$
    \end{enumerate}
\end{otherl}
Then Khinchine`s inequality is the following.
\begin{otherl}\label{Khinchine`s}
(Khinchine`s inequality). For $0<p<\infty$, there exists a constant $C_p$such that
$$C_p^{-1} \left(\sum_{k=1} ^n|\lambda_k|^2\right)^{p/2} \leq \int_0 ^1 \left|\sum_{k=1} ^n \lambda_k R_k (t)\right|^p dt \leq C_p \left( \sum_{k=1} ^n |\lambda_k|^2\right)^{p/2},$$
for all $n\in \mathbb{N}$ and $\left\{\lambda_k\right\}_{k=1} ^n \subset \mathbb{C}.$
\end{otherl}
\begin{otherl}\label{embed2}
   Let $\omega \in \mathcal{W}$ and $0<p<q<\infty$. Let $\mu$ be a finite positive Borel measure on $\mathbb{D}$. Then, the following statements are equivalent:
    \begin{enumerate} 
    \item[(a)] The operator $I_\mu : S_\omega ^p \longrightarrow L^q (\mathbb{D} ,d\mu)$ is bounded.
    \item[(b)] The operator $I_\mu : S_\omega ^p \longrightarrow L^q (\mathbb{D} ,d\mu)$ is compact.
    \item[(c)] The function
    $$F_\delta,\mu (\varphi)\in L^{p/(p-q)} (\mathbb{D},dA).$$
    \end{enumerate}
\end{otherl}
Also, we state the results in the case $0 < q < \infty$ and $p =\infty$ as follows:
\begin{lemma} \label{lem3} 
\begin{enumerate} Let $\omega \in \mathcal{W}$ and    $0<q< p= \infty.$ Let $\mu$ be a positive Borel measure on $\D.$ Then, the following statements qre equivalent:
\item The operator $I_{\mu}: S^p_\omega\rightarrow L^q(\D,d\mu)$  is bounded.
\item The operator $I_{\mu}: S^p_\omega\rightarrow L^q(\D,d\mu)$  is compact.
\item The function 
\begin{equation}\label{eq3}
F_{\delta,\mu}(\varphi)(z)\in L^{1}(\D,dA).
\end{equation}
\end{enumerate}
\end{lemma}
\subsection{Schatten class operators}
 For a positive compact operator $T$ on a separable Hilbert
space $H$, there exist orthonormal sets $\{e_k\}$ in $H$ such that
$$Tx = \sum_k \lambda_k \langle x, e_k\rangle, \quad x \in H,$$
where the points ${\lambda_k}$ are nonnegative eigenvalues of $T$. This is referred to as the
canonical form of a positive compact operator $T$. For $0 < p < \infty$, a compact
operator $T$ belongs to the Schatten class $\mathcal{S}_p$ on H if the sequence ${\lambda_k}$ belongs to
the sequence space $\ell^p$,
$$\|T\|_{S_p}^p =\sum_k |\lambda_k|^p <\infty.$$
When $1 \leq p < \infty$, $\mathcal{S}_p$ is the Banach space with the above norm and $\mathcal{S}_p$ is a metric
space when $0 < p < 1$. In general, if $T$ is a compact linear operator on $H$, we say
that $T \in \mathcal{S}_p$ if $(T^*T)^{p/2} \in \mathcal{S}_1, 0 < p < \infty$. Moreover,
$$(T^*T)^{p/2} \in \mathcal{S}_1 \iff T^*T \in \mathcal{S}_{p/2}.$$

\section{Boundedness and compactness}\label{Section 3}

In this section we first provide the proof of Theorem \ref{Theorem 1.1} and then show how our results are related to the results of Constantin and Pel\'{a}ez \cite{Co-Pe} on Fock spaces and of Pau and Pel\`{a}ez \cite{PP1} on Bergman spaces.

\subsection{Proof of Theorem 1.1}
  \textbf{(A) Boundedness.} For $0<p\leq q<\infty$, suppose that the operator $C_{(\psi,g)} :A_\omega ^p \longrightarrow A_\omega^q$ is bounded.
   Then, by \eqref{eqnorm}, we have 
   \begin{displaymath}\label{eq1}
   \begin{split}
    \|C_{(\psi,g)} f\|_{A_\omega ^q}^q &\asymp\int_\mathbb{D} |f^\prime(\psi(z))|^q\,|g(\psi(z))|^q\,|\psi^\prime (z)|^q \,\frac{\omega(z)^\frac{q}{2}}{(1+\varphi^\prime(z))^q}\, dA(z)\\
   &=\int_\mathbb{D} |f^\prime (z)|^q \,d\mu_{\psi, \omega,g}= \|f^\prime\|_{L^q (\mathbb{D},d_{\mu_{\psi,\omega,g}})}^{q}.
   \end{split}
  \end{displaymath}
  Hence, $C_{(\psi,g)} :A_\omega ^p \longrightarrow A_\omega^q$ is bounded if and only if $I_\mu : S_\omega ^p \longrightarrow L^q (\mu_\psi,\omega,g)$ is bounded. Using (i) of Lemma \ref{embed2}, this is equivalent to
   $$\sup_{z\in\mathbb{D}} \frac{1}{\tau(z)^{2q/p}}\int_{D_\delta(z)} (1+\varphi^\prime(\xi))^q\, \omega(\xi)^{-q/2} \,d\mu_{\psi,\omega,g}(\xi) <\infty.$$
   By Theorem \ref{thm:CMPQ}, this equivalent to
   $$\sup_{z\in \mathbb{D}} \tau(z)^{2(1-q/p)} \int_\mathbb{D} |k_{q,z} (\xi)|^q \,\omega(\xi)^{q/2}\, d\nu_{\psi,\omega,g}(\xi) <\infty.$$
   Then, by Lemma \ref{embed1}, we obtain
  \begin{displaymath}
   \begin{split}
    \tau(z)^{2(1-q/p)} \int_\mathbb{D} |k_{q,z} (\xi)|^q\, &\omega(\xi)^{q/2} \,d\nu_{\psi,\omega,g}(\xi) \asymp \int_\mathbb{D} |k_{p,z}(\xi)|^q\,\omega(\xi)^{q/2} \,d\nu_{\psi,\omega,g}(\xi)\\
   &=\int_{\mathbb{D}} |k_{p,z} (\psi(\xi))|^q \,|g(\psi(\xi))|^q \,|\psi^\prime(\xi)|^q\, \frac{(1+\varphi^\prime(\psi(\xi)))^q}{(1+\varphi^\prime(\xi))^q} \,\omega(\xi)^{q/2}\, dA(\xi)=M_{1,p,q}^\psi.\label{bdd first case}   
   \end{split}
   \end{displaymath}
   Thus, $C_{(\psi, g)}$ is bounded if and only if $M_{1,p,q}^\psi g(z)\in L^\infty (\mathbb{D},dA).$
   
   The proof that $C_\psi^g$ is bounded if and only if $M^\psi_{0,p,q}(g)\in L^{\infty}(\D,dA)$ follows in a similar fashion.
   
   \textbf{Compactness.} For $0<p\leq q<\infty,$ considering that operator $C_{(\psi,g)} :A_\omega ^p \longrightarrow A_\omega^q$ is compact. Then, by (ii) of Lemma \ref{embed1}, can be used in a similar way for proving compactness.
This means, $\lim_{|z|\rightarrow 1^-}M_{1,p,q}^\psi (g)=0.$

Now suppose that the operator $C_{g^\psi}$ is compact. By \eqref{eqnorm}  
\begin{equation}\label{norm 2nd operator}
 \|C_g^\psi f\|^q_{A^q_{\om}}=\int_{\D}\frac{|f(\psi(z))|^q\,|g(\psi(z))|^q\,|\psi^\prime(z)|^q}{(1+\varphi'(z))^q}\,\,\omega(z)^{\frac{q}{2}}\, dA(z)=\|f\|^q_{L^q(\mu_{\phi,\omega,g})}.   
\end{equation}
 We conclude that $C_g^\psi: A^p_\omega\rightarrow A^q_\omega$ is compact if and only if the measure $\nu_{\phi,\omega,g}$ is a vanishing $q$-Carleson measure for $A_\omega^p$. 
This is equivalent to 
$$
\lim_{|z|\to1^-}\,\tau(z)^{2(1-q/p)}\int_{\D} |k_{q,z}(\xi)|^q\,\omega(\xi)^{q/2}\,d\nu_{\psi,\omega,g}(\xi)=0.
$$
Now, using $(a)$ of Lemma \ref{lem:RK-PE1}, we obtain 
\begin{align*}\label{maincond1}
\tau(z)^{2(1-q/p)}\int_{\mathbb{D}} |K_{q,z}(\xi)|^q\, \omega(\xi)^{q/2} \,dv_{\psi,\omega,q}(\xi)&\asymp \int_{\D}|k_{p,z}(\xi)|^q\, \omega(\xi)^{q/2}\,d\nu_{\psi,\omega,q}(\xi)\\& \asymp    \int_{\mathbb{D}} |k_{p,z}(\psi(\xi))|^q \,  \frac{|g(\psi(z))|^q\,|\psi^\prime(z)|^q}{(1+{\varphi}^{'}(\xi))^q}\,\omega(z)^{q/2}\,dA(\xi)=M^{\psi}_{0,p,q}.
\end{align*}
Therefore, $\lim_{|z|\to1^-}M^\psi_{0,p,q}(g)=0$ if and only if the operator  $C_g^\psi$ is compact.

\textbf{(B)} Suppose that $C_{(\psi,g)}$ is bounded and let $\{f_n \}\subset A_\omega^p$ be a bounded sequence converging to zero uniformly on compact subsets of $\mathbb{D}.$ Now, replacing $f$ by $f_n$ in \eqref{eqnorm}, we get
\begin{equation}\label{eq33}
  \begin{aligned}
      \| C_{(\psi,g)} f_n \| = \| f_n^\prime \|_{L^q(\mu_{\psi,\omega,g})}^q,
  \end{aligned}  
\end{equation}
by compactness of the embedding operator $I_{\mu_{\psi,\omega,g}},$ in Lemma \ref{embed2}, we have
$$\| C_{(\psi,g)} f_n \|_{A_\omega^q}^q \rightarrow 0, \quad\text{as}\qquad n\rightarrow \infty,$$
we obtain the compactness of the operator $C_{(\psi,g)}.$
 Now, when $p<\infty$ we prove that boundedness is equivalent to compactness. Using \eqref{eq33} and Lemma \ref{embed2}, we get $C_{(\psi,g)}$ is bounded if and only if $I_{\mu_{\psi,\omega,g}}: S_\omega^p \rightarrow L^q (\mu_{\psi,\omega,g})$ is bounded if and only if $I_{\mu_{\psi,\omega,g}}: S_\omega^p \rightarrow L^q (\mu_{\psi,\omega,g})$ is compact if and only if the function
$$F_{\delta,\mu_{\phi,\omega,g}}(\varphi)(z):= \frac{1}{\tau(z)^{2q/p}}\int_{D_{\delta}(z)} (1+\varphi'(\xi))^q\omega(\xi)^{-q/2}\,d\mu_{\psi,\omega,g}(\xi)m$$ belongs to $ L^{p/(p-q)}(\D,dA).$ According to Theorem \ref{thm:CMOP}, this is equivalent to 
$$
\int_{\D}|k_{q,z}(\xi)|^q\, \omega(\xi)^{q/2}\,d\nu_{\psi,\omega,g}(\xi) \in  L^{p/(p-q)}(\D,dA),
$$
which is equivalent to $M^\psi_{1,p,q}(g)(z)\in L^{p/(p-q)}(\D,d\lambda),$ where $d\lambda(z)=dA(z)/\tau(z)^2,$ because of 
\begin{displaymath}
\begin{split}
 \int_{\D}  G_{q}(\nu^q_{\psi,\omega,g})^{p/p-q}\,dA(z)& = \int_{\D}\Big( \tau(z)^{2(1-q/p)}\, G_{q}(\nu^q_{\psi,\omega,g})\Big)^{\frac{p}{p-q}}\,d\lambda(z)\\& \asymp \int_{\D}\Big( \tau(z)^{2(1-q/p)}\, \int_{\D} |k_{p,z}(\xi)|^q\,\omega(\xi)^{q/2}\,d\nu_{\psi,\omega,g}(\xi)\Big)^{\frac{p}{p-q}}\,d\lambda(z)\\& =\int_{\D}\Big( \tau(z)^{2(1-q/p)}\, \int_{\D} |k_{p,z}(\xi)|^q\,(1+\varphi'(\xi))^q\,d\mu_{\psi,\omega,g}(\xi)\Big)^{\frac{p}{p-q}}d\lambda(z)  \\&=  \int_{\D} M^\psi_{1,p,q}(g)(z)^{p/p-q}\,d\lambda(z),
\end{split}
\end{displaymath}
which proves boundedness of the operator $C_{(\psi,g)}: A^p_\omega\rightarrow A^q_\omega$.

Now, for $0<q<p=\infty,$ we suppose that $C_{(\psi,g)}: A^\infty_\omega\rightarrow A^q_\omega$ is bounded. Let $\{f_n\}\subset A^\infty_\omega$ be a bounded sequence converging  to zero uniformly on compact subsets of $\D$, we get
\begin{equation}\label{eq34}
  \begin{aligned}
      \| C_{(\psi,g)} f_n \| = \| f_n^\prime \|_{L^q(\mu_{\psi,\omega,g})}^q,
  \end{aligned}  
\end{equation}
by compactness of the embedding operator $I_{\mu_{\psi,\omega,g}},$ in Lemma \ref{embed2}, we obtain the compactness of the operator $C_{(\psi,g)}.$
 Now, we prove that boundedness is equivalent to $M_{1,p,q}^\psi (g)\in L^s(\mathbb{D},d\lambda)$ when $p=\infty.$ By \eqref{eq34} and Lemma \ref{lem3}, we get $C_{(\psi,g)}$ is bounded if and only if $I_{\mu_{\psi,\omega,g}}: S_\omega^p \rightarrow L^q (\mu_{\psi,\omega,g})$ is bounded if and only if $I_{\mu_{\psi,\omega,g}}: S_\omega^p \rightarrow L^q (\mu_{\psi,\omega,g})$ is compact if and only if the function
$$F_{\delta,\mu_{\phi,\omega,g}}(\varphi)(z):= \frac{1}{\tau(z)^{2}}\int_{D_{\delta}(z)} (1+\varphi'(\xi))^q\,\omega(\xi)^{-q/2}\,d\mu_{\psi,\omega,g}(\xi),$$ belongs to $ L^{p/(p-q)}(\D,dA).$ According to Theorem \ref{thm:CMOP2}, this is equivalent to 
$$
\int_{\D}|k_{q,z}(\xi)|^q\, \omega(\xi)^{-q/2}\,d\nu_{\psi,\omega,g}(\xi) \in  L^1(\D,dA),
$$
which is equivalent to $M^\psi_{1,p,q}(g)(z)\in L^1(\D,d\lambda),$ where $d\lambda(z)=dA(z)/\tau(z)^2.$ Because of 
\begin{displaymath}
\begin{split}
M^\psi_{1,p,q}\asymp \tau(z)^{2(1-q/p)}\,\int_\mathbb{D}|k_{p,z}(\xi)|^q\,\omega(\xi)^{q/2}\,d\nu^q_{\psi,\omega,g}(\xi).
\end{split}
\end{displaymath}

Let $ 0<q< p< \infty$ and suppose that $C_g^\psi: A^p_\omega\rightarrow A^q_\omega$ is bounded. To prove that $C_g^\psi$ is compact, notice first that \eqref{norm 2nd operator} implies that the measure $\nu_{\psi,\omega,g}$ is a $q$-Carleson measure for $A_\omega^p$. Thus, by Theorem \ref{thm:VCMQP},  $\nu_{\psi,\omega,g}$ is a vanishing $q-$Carleson measure for $A_\omega^p$. By Theorem \ref{Theorem 1.1} (A), we have 
$$
\lim_{n\rightarrow\infty}\|C_g^\psi f_n\|^q_{A^q_{\om}}=0
$$
for any sequence $\{f_n\}\subset A^p_\omega$ that converges  to zero uniformly on compact subsets of $\D.$ Now Lemma $3.7$ of \cite{tija} shows that $C_g^\psi$ is compact.

Next we prove that boundedness is equivalent to $M_{0,p,q}^\psi (g)\in L^s(\mathbb{D},d\lambda)$ when $p<\infty$. Assume first that $M_{0,p,q}^\psi (g)\in L^s(\mathbb{D},d\lambda)$. Then
\begin{equation}\label{equ111}
\begin{split}
\int_{\D} G_{q}(v_{\psi,\omega,q})(z)^{p/(p-q)} dA(z)&=\int_{\D} \Big(\tau(z)^{2(1-\frac{q}{p})}G_{q}(v_{\psi,\omega,q})(z)\Big)^{p/(p-q)}d\lambda(z)\asymp \int_{\D}M^\psi_{0,p,q}(g)^{p/(p-q)}d\lambda(z).
\end{split}
\end{equation}
According to Theorem \ref{thm:CMOP}, $\nu_{\psi,q}$ is a $q$-Carleson measure for $A_{\omega}^p$. Then, by (\ref{eqnorm}) for any function $f\in A_{\omega}^p$, we get
$$
\|C_g^{\psi} \,f_n\|_{A_\omega^q}^q\asymp\int_{\D} |f(z)|^q\, {\omega(z)^{q/2}}\,d\nu_{\psi,\omega,g}(z)\lesssim \|f\|_{A_{\omega}^p}^q.
$$
Therefore, the operator $C_g^\psi$ is bounded.

Conversely, assume that the operator $C_g^\psi: A^p_\omega\rightarrow A^q_\omega$ is bounded. Then, we have
 $$
 \|C_g^\psi \,f\|_{A_{\omega}^q}^q \asymp\int_{\D} |f(z)|^q\, {\omega(z)^{q/2}}\,d\nu_{\psi,\omega,g}(z), \quad \text{for any function}\quad f\in A_{\omega}^p.
 $$
 This together with our assumption, implies that the measure $\nu_{\psi,\omega,g}$ is a $q-$Carleson measure for $A_\omega^p$. According to Theorem \ref{thm:CMOP}, $\nu_{\psi,\omega,g}$ belongs to $L^{p/(p-q)}(\D,dA)$. Combining this with \ref{equ111}, we conclude that $M^\psi_{0,p,q}(g)\in L^{p/(p-q)}(\D,d\lambda).$
 
 Let $0<q<p=\infty$ and suppose that $C_g^\psi: A^\infty_\omega\rightarrow A^q_\omega$ is bounded, that is, for any  function $f\in A_\omega^p$, we have 
\begin{equation*}
\begin{split}
\|C_g^\psi f\|^q_{A^q_{\om}}=\int_{\D}\frac{|f(\psi(z))|^q|g(\psi(z))|^q|\psi^\prime(z)|^q}{(1+\varphi'(z))^q}\,\,\omega(z)^{\frac{q}{2}}\, dA(z) \lesssim \|f\|_{A_\omega^\infty}^q,
\end{split}
\end{equation*}
We show that $C^\psi_g$ is compact. Using Theorem \ref{thm:CMOP2}, conclude that the measure  $\nu_{\psi,\omega,g}$ is a $q-$Carleson measure for $A_\omega^\infty$. Therefore, by Theorem \ref{thm:VCMQP},  $\nu_{\psi,\omega,g}$ is a vanishing $q-$Carleson measure for $A_\omega^\infty$. As in the previous case, this shows the compactness of the operator $C_g^\psi.$

Next we prove that boundedness and $M_{0,p,q}^\psi (g)\in L^s(\mathbb{D},d\lambda)$ are equivalent when $p=\infty$. First, we assume that that the condition $M_{0,p,q}^\psi (g)\in L^s(\mathbb{D},d\lambda)$ holds. Then
\begin{equation}\label{eq111}
\begin{split}
\int_{\D} G_{q}(v_{\psi,\omega,q})(z) \,dA(z)=\int_{\D} \Big(\tau(z)^{2(1-\frac{q}{p})}G_{q}(v_{\psi,\omega,q})(z)\Big)\,d\lambda(z) \asymp \int_{\D}M^\psi_{0,p,q}(g)\,d\lambda(z).
\end{split}
\end{equation}
According to Theorem \ref{thm:CMOP2}, $\nu_{\psi,q}$ is a $q-$Carleson measure for $A_{\omega}^\infty$. Then for any function $f\in A_{\omega}^\infty$, we get
$$
\|C_g^\psi\,f_n\|_{A_{\omega}^{q}}^q\asymp \int_{\D} |f(z)|^q\, {\omega(z)^{q/2}}\,d\nu_{\psi,\omega,g}(z)\lesssim \|f\|_{A_{\omega}^\infty}^q.
$$
Thus, the operator $C_g^\psi$ is bounded.

Conversely, suppose the operator $C_g^\psi: A^\infty_\omega\rightarrow A^q_\omega$ is bounded. Then, for any function  $f\in A_{\omega}^\infty$, we have
 $$
 \|C_g^\psi \,f\|_{A_{\omega}^q}^q=\int_{\D} |f(z)|^q\, {\omega(z)^{q/2}}\,d\nu_{\psi,\omega,g}(z).
 $$
 This together with our assumption, implies that the measure $\nu_{\psi,\omega,g}$ is a $q$-Carleson measure for $A_\omega^\infty$. According to Theorem \ref{thm:CMOP2}, $\nu_{\psi,\omega,g}$ belongs to $L^{1}(\D,dA)$. Combining this with (\ref{eq111}), we conclude that $M^\psi_{0,p,q}(g)\in L^{1}(\D,d\lambda).$

\textbf{(C) Boundedness.} For $0<p<\infty$, assume that the equation \eqref{bdd 2nd case} holds. Next, using our assumption and \eqref{eqnorm1}, we have
   \begin{equation}
   \begin{aligned}
       \|C_{(\psi,g)} f\|_{A_\omega^\infty} 
    &\asymp\sup_{z\in\mathbb{D}} |f^\prime(\psi(z))|\,|g(\psi(z))|\,|\psi^\prime(z)|\,\frac{\omega(z)^{1/2}}{(1+\varphi^\prime(z))}\\
   &\leq \sup_{z\in\mathbb{D}} N_{1,p,\infty}^{\psi,1/p} (g)(z)\, \sup_{z\in \mathbb{D}} \frac{|f^\prime (\psi(z))|\,\omega(\psi(z))^{1/2}}{(1+\varphi^\prime(\psi(z)))}\,\triangle\varphi(\psi(z))^{-1/p}\\
   &\leq \sup_{z\in\mathbb{D}} N_{1,p,\infty}^{\psi,1/p}(g)(z)\, \sup_{z\in \mathbb{D}} \frac{|f^\prime (\psi(z))|\,\omega(\psi(z))^{1/2}}{(1+\varphi^\prime(\psi(z)))}\, \tau(\psi(z))^{2/q}.
   \end{aligned}
   \end{equation}
   By Lemma \ref{DERI}, we get
   \begin{equation}
       \begin{aligned}
 \|C_{(\psi,g)} f\|_{A_\omega^\infty} 
 &\lesssim\sup_{z\in\mathbb{D}} \left(\int_{D_\delta (\psi(z))} 
  \frac{|f^\prime(\xi)|^p\, \omega(\xi)^{p/2}}{(1+\varphi^\prime(\xi))^p}\,dA(\xi)\right)^{1/p}\\
  &\leq \left(\int_\mathbb{D}  \frac{|f^\prime(\xi)|^p \,\omega(\xi)^{p/2}}{(1+\varphi^\prime(\xi))^p}\,dA(\xi)\right)^{1/p}\lesssim \|f\|_{A_\omega ^p},
       \end{aligned}
   \end{equation}
which implies that $C_{(\psi,g)}$ is bounded.

Conversely, we suppose that $C_{(\psi,g)} :A_\omega ^p \longrightarrow A_\omega^\infty$ is bounded. Taking $\xi\in\mathbb{D}$ such that $|\psi(\xi)|>\rho_0$, we consider the function $f_{\psi(\xi),n,p}$ given by $f_{\psi(\xi),n,p}:=\frac{F_{\psi(\xi),n,p}}{\tau(\psi(\xi))^{2/p}},$ where $F_{\psi(\xi)),n,p}$ is the test function in Lemma~\ref{Borichevlemma}. Notice that $f_{\psi(\xi),n,p}\in A_\omega ^p$ with $\|f_{\psi(\xi),n,p}\|\asymp 1.$ By our assumption, we get
\begin{equation}
    \begin{aligned}
    \infty &>\|C_{(\psi,g)} {(f_{\psi(\xi),n,p})}\|_{A_\omega ^\infty}\geq \sup_{z\in\mathbb{D}} \frac{|f_{\psi(\xi),n,p}^\prime (\psi(z))|\,|g(\psi(z)|\,|\psi^\prime(z)|}{(1+\varphi^\prime(z))} \,\omega(z)^{1/2}\\
    &\geq \sup_{z\in\mathbb{D}}\frac{|F_{\psi(\xi),n,p}^\prime (\psi(z))|\,|g(\psi(z)|\,|\psi^\prime(z)|}{\tau(\psi(\xi))^{1/2}(1+\varphi^\prime(z))}\, \omega(z)^{1/2}\geq \sup_{\xi\in\mathbb{D}}\frac{|F_{\psi(\xi),n,p}^\prime (\psi(\xi))|\,|g(\psi(\xi)|\,|\psi^\prime(\xi)|}{\tau(\psi(\xi))^{1/2}(1+\varphi^\prime(\xi))} \,\omega(\xi)^{1/2}.
    \end{aligned}
\end{equation}
Now, by Lemma (\ref{BL3}),
$$|F_{\psi(\xi),n,p}^\prime (z)|\,\omega(z)^{1/2}\asymp (1+\varphi^\prime(z)), \quad z\in D_\delta (\psi(\xi)),$$
so we obtain
\begin{equation}
\begin{aligned}
    \infty>\|C_{(\psi,g)} {(f_{\psi(\xi),n,p})}\|_{A_\omega ^\infty}&\geq
     |g(\psi(\xi))|\,|\psi^\prime(\xi)|\,\frac{(1+\varphi^\prime(\psi(\xi)))}{(1+\varphi^\prime(\xi))}\,\frac{\omega(\xi)^{1/2}}{\omega(\psi(\xi))^{1/2}}\tau(\psi(\xi))^{-2/p}\\
     &\asymp  |g(\psi(\xi))|\,|\psi^\prime(\xi)|\,\frac{(1+\varphi^\prime(\psi(\xi)))}{(1+\varphi^\prime(\xi))}\,\frac{\omega(\xi)^{1/2}}{\omega(\psi(\xi))^{1/2}}\,\triangle\varphi(\psi(\xi))^{1/p}\label{bdd 2}= N_{1,p,\infty}^{\psi,1/p}(\xi).
\end{aligned}
\end{equation}
On the other hand, by taking $f(z)=z$ and using the boundedness of the operator $C_{(\psi,g)} :A_\omega ^p \longrightarrow A_\omega^\infty,$ we get
$$\|C_{(\psi,g)}\|_{A_\omega ^\infty} =\sup_{z\in\mathbb{D}}|g(\psi(z))|\,|\psi^\prime(z)|\,\frac{\omega(z)^{1/2}}{(1+\varphi^\prime(z))}\lesssim \|f\|_{A_\omega ^p} <\infty.$$
Hence, in the case of $|\psi(\xi)|\leq\rho_0, \xi \in\mathbb{D}$, we have
\begin{align*}
   |N_{1,p,\infty}^{\psi,1/p} (\xi)| &=|g(\psi(\xi))|\,|\psi^\prime(\xi)|\,\frac{(1+\varphi^\prime(\psi(\xi)))}{(1+\varphi^\prime(\xi))}\,\frac{\omega(\xi)^{1/2}}{\omega(\psi(\xi))^{1/2}}\,\triangle\varphi(\psi(\xi))^{1/p}\\
    &\asymp |g(\psi(\xi))|\,|\psi^\prime(\xi)|\,\frac{(1+\varphi^\prime(\psi(\xi)))}{(1+\varphi^\prime(\xi))}\,\frac{\omega(\xi)^{1/2}}{\omega(\psi(\xi))^{1/2}}\,\tau(\psi(\xi))^{-2/p}\leq R_1 \,|g(\psi(\xi))|\,|\psi^\prime(\xi)|\frac{\omega(\xi)^{1/2}}{(1+\varphi^\prime(\xi))}<\infty,  
\end{align*}
where $R_1= \sup_{|\psi(\xi)|\leq\rho_0}\left\{(1+\varphi^\prime(\psi(\xi)))\,\omega(\psi(\xi))^{-1/2} \,\tau(\psi(\xi))^{-2/p}\right\}<\infty.$
It remains to combine this with \eqref{bdd 2}.
The case $p=\infty$ can be proved in a similar manner. 

\textbf{Compactness.} For $0<p<\infty$, assume that the operator $C_{(\psi,g)} :A_\omega ^p \longrightarrow A_\omega^\infty$ is compact. Then, $f_{\psi(\xi),n,p}\, \text{belongs to}\, A_\omega^p$ and converges to zero uniformly on compact subsets of $\mathbb{D}$ as $|\psi(\xi)|\rightarrow 1$ (see Lemma 3.1 in \cite{PP1}), so 
 $\| C_{(\psi,g)} (f_{\psi(\xi),n,p}) \|_{A_\omega^\infty}\rightarrow 0$  when $|\psi(\xi)|\rightarrow 1.$
 Now, by \eqref{bdd 2},
 $$ 0=\lim_{|\psi(\xi)|\rightarrow {1^-}} \| C_{(\psi,g)} (f_{\psi(\xi),n,p}) \| \gtrsim \lim_{|\psi(\xi)|\rightarrow {1^-}} N_{g,\psi,1/p} (\xi),$$
 this achieves the desired result.
 
In contrast, let $\{f_n\}$ be a bounded sequence of function in $A_\omega ^p$ converging to zero uniformly on compact subset of $\mathbb{D}.$ Since compactness condition in (C) holds, for any $\varepsilon > 0,$ there exists $r_0 >0$ such that
  $$N_{1,p,\infty}^{\psi,1/p}(g)(\xi) = |g(\psi(\xi))|\,|\psi^\prime(\xi)|\,\frac{(1+\varphi^\prime(\psi(\xi)))}{(1+\varphi^\prime(\xi))}\,\frac{\omega(\xi)^{1/2}}{\omega(\psi(\xi))^{1/2}} \,\triangle\varphi(\psi(\xi))^{1/p}<\varepsilon,$$
where $|\psi(\xi)|>r_0.$ Then, by \eqref{Eq-gamma}, we have
\begin{equation}
\begin{split}
   & \frac{|f_n^\prime (\psi(\xi))|\,|g(\psi(\xi))|\,|\psi^\prime(\xi)|}{(1+\varphi^\prime(\xi))}\,\omega(\xi)^{1/2}
    \\ & \lesssim\left( \frac{1}{\tau(\psi(\xi))^2} \int_{D_\delta (\psi(\xi))}\frac{|f_n ^\prime (\psi(s))|^p}{(1+\varphi^\prime(\psi(s)))^p}\,\omega(\psi(s))^{p/2}\,\triangle\varphi(\psi(\xi))dA(s)) \right)^{1/p} N_{1,p,\infty}^{\psi,1/p}(\xi)\\
   & \lesssim \| f_n \|_{A_\omega^p} N_{1,p,\infty}^{\psi,1/p} (\xi) <\varepsilon. \label{eq epsalon}
\end{split}
\end{equation}
For $|\psi(\xi)|\geq r_0$, we have
$$\sup_{|\psi(\xi)|\leq r_0} \frac{|f^\prime(\psi(\xi))|\,|g(\psi(\xi))|\,|\psi^\prime(\xi)|}{(1+\varphi^\prime(\xi))}\,\omega(\xi)^{1/2} \lesssim \sup_{|\psi(\xi)|\leq r_0}|f^\prime (\psi(\xi))|\rightarrow 0,\quad \text{as}\,\quad n\rightarrow \infty,$$
and also the sequence of function $f_n ^\prime$ converges to zero uniformly on compact subset of $\mathbb{D}$, see Lemma (\ref{DERI}). Combining this with \eqref{eq epsalon} gives
$$\| C_{(\psi,g)} (f_n) \|_{A_\omega^\infty} \asymp \frac{|f_n ^\prime(\psi(z))|\,|g(\psi(z))|\,|\psi^\prime(z)|}{(1+\varphi^\prime(z))}\,\omega(z)^{1/2}\rightarrow 0, \quad as\quad n\rightarrow \infty,$$
which means that the operator $C_{(\psi,g)}: A_\omega^p \rightarrow A_\omega^\infty$ is compact.

The case $p=\infty$ can be proved similarly. Also, the proof of boundedness and compactness of the operator $C_g^\psi $ when $0<p\leq\infty$ is similar to that of the operator $C_{\psi,g}$, and hence we omit the details.

\subsection{Additional results on boundedness and compactness}

The next result gives a necessary condition for the operator $C_{(\psi,g)}: A^p_\omega\rightarrow A^q_\omega$ to be bounded or compact when $0<p,q<\infty.$

\begin{proposition} Let $ 0<p, q< \infty.$ Suppose that  $\omega \in \mathcal{W},$  $\psi$ is an analytic self-map of $\D,$ and $g$ is an analytic function on $\D.$ 
\begin{enumerate} 
\item[{\rm (i)}]	 The  operator $C_{(\psi,g)}: A^p_\omega\rightarrow A^q_\omega$ is bounded, then  
\begin{equation}\label{eq-c1}
\sup_{z\in \mathbb{D}}|g(z)||\psi^\prime(z)|\frac{\tau(z)^{2/q}}{\tau(\psi(z))^{2/p}}\frac{(1+\varphi'(\psi(z))}{(1+\varphi'(z))}\frac{\omega(z)^{1/2}}{\omega(\psi(z))^{1/2}}<\infty.
\end{equation}

\item[{\rm (ii)}]The  operator $C_{(\psi,g)}: A^p_\omega\rightarrow A^q_\omega$ is compact, then
\begin{equation}\label{eq-nn2}
\lim_{|\psi(z)|\to1^-} 
|g(z)||\psi^\prime(z)|\frac{\tau(z)^{2/q}}{\tau(\psi(z))^{2/p}}\frac{(1+\varphi'(\psi(z))}{(1+\varphi'(z))}\frac{\omega(z)^{1/2}}{\omega(\psi(z))^{1/2}}=0.
\end{equation}
\end{enumerate}
\end{proposition}
\begin{proof}
Let us start first by proving that $(i)$. Suppose that the operator $C_{(\psi,g)}$ is bounded. Let $\xi\in\D$ such that $|\psi(\xi)|>\rho_0,$ we consider $ F_{\psi(\xi),n}$ is the test function defined in Lemma \ref{Borichevlemma}. By \eqref{Eq-gamma}, we have 
\begin{equation*}
\begin{split}
\|F_{\psi(\xi),n,p}\|^q_{A^p_{\om}} &\gtrsim\|C_{(\psi,g)}F_{\psi(\xi),n,p}\|^q_{A^q_{\om}}=\int_{\D}\frac{|F'_{\psi(\xi),n,p}(\psi(z))|^q}{(1+\varphi'(z))^q}\,|g(\psi((z))|^q |\psi^\prime (z)|^q\,\omega(z)^{\frac{q}{2}}\, dA(z)\\
&\gtrsim \tau(\xi)^2\frac{|F'_{\psi(\xi),n,p}(\psi(\xi))|^q}{(1+\varphi'(\xi))^q}\,|g(\psi(\xi))|^q |\psi^\prime (\xi)|^q\,\omega(\xi)^{\frac{q}{2}}.
\end{split}
\end{equation*}
Using Lemma \ref{BL3}, we get
\begin{equation*}
\begin{split}
\|F_{\psi(\xi),n,p}\|^q_{A^p_{\om}} 
&\gtrsim \tau(\xi)^2\,|g(\psi(\xi))|^q \,|\psi^\prime (\xi)|^q\,\frac{(1+\varphi'(\psi(\xi)))^q}{(1+\varphi'(\xi))^q}\,\frac{\omega(\xi)^{\frac{q}{2}}}{\omega(\psi(\xi)^{\frac{q}{2}}}.
\end{split}
\end{equation*}
 By Lemma \ref{Borichevlemma}, we obtain 
\begin{equation}\label{eqp1}
\begin{split}
1\gtrsim|g(\psi(\xi))|\,|\psi^\prime (\xi)|^q\, \frac{\tau(\xi)^{2/q}}{\tau(\psi(\xi))^{2/p}}\,\frac{(1+\varphi'(\psi(\xi)))}{(1+\varphi'(\xi))}\,\frac{\omega(\xi)^{\frac{1}{2}}}{\omega(\psi(\xi))^{\frac{1}{2}}}.
\end{split}
\end{equation}
On the other hand, for $|\psi(\xi)|\leq \rho_0$, we have 
$$
\sup_{\psi(\xi)\le\rho_0 }|g(\psi(\xi))|\,|\psi^\prime (\xi)|\, \frac{\tau(\xi)^{2/q}}{\tau(\psi(\xi))^{2/p}}\,\frac{(1+\varphi'(\psi(\xi)))}{(1+\varphi'(\xi))}\,\frac{\omega(\xi)^{\frac{1}{2}}}{\omega(\psi(\xi))^{\frac{1}{2}}}<\infty.
$$
This together with \eqref{eqp1}, completes the proof of (i).

Now, we prove (ii). Suppose that the operator $C_{(\psi,g)}$ is compact. Taking $\xi\in\D$ such that $|\psi(\xi)|>\rho_0$ and  the  bounded sequence 
$$\Big\{ f_{\psi(\xi),n,p}:=\frac{F_{\psi(\xi),n,p}}{\tau(\psi(\xi))^{2/p}},\quad\text{for}\quad |\psi(\xi)|>\rho_0\Big\} $$
of $A^p_\omega$ that converges uniformly to zero on compact subsets of $\D$ as $|\psi(\xi)|\to 1.$ By \eqref{Eq-gamma} and Lemma \ref{BL3}, we get 
\begin{equation*}
\begin{split}
\|C_{(\psi,g)}\,f_{\psi(\xi),n,p}\|^q_{A^q_{\om}}
&=\int_{\D}\frac{|f'_{\psi(\xi),n,p}(\psi(z))|^q}{(1+\varphi'(z))^q}\,|g(\psi(z))|^q |\psi^\prime (z)|^q\,\omega(z)^{\frac{q}{2}}\, dA(z)\\
&\gtrsim \tau(\xi)^2\,\frac{|f'_{\psi(\xi),n,p}(\psi(\xi))|^q}{(1+\varphi'(\xi))^q}\,|g(\psi(\xi))|^q |\psi^\prime (\xi)|^q\,\omega(\xi)^{\frac{q}{2}}\\
&\gtrsim|g(\psi(\xi))|^q\,|\psi^\prime (\xi)|^q \frac{\tau(\xi)^2}{\tau(\psi(\xi))^{2q/p}}\frac{(1+\varphi'(\psi(\xi)))^q}{(1+\varphi'(\xi))^q}\,\frac{\omega(\xi)^{\frac{q}{2}}}{\omega(\psi(\xi))^{\frac{q}{2}}}.
\end{split}
\end{equation*}
Since $C_{(\psi,g)}$ is compact, we obtain the desired result and the proof is complete.
\end{proof}

 Consequently, in the next result, we show that the more useful, necessary conditions for the boundedness and compactness of the operators $C_{\psi,g}$ and $C_g^\psi$ for any $0<p,q<\infty.$
 
 \begin{proposition} Let $ 0<p, q< \infty.$ Suppose that  $\omega \in \mathcal{W},$  $\psi$ is an analytic self-map of $\D,$ and $g$ is an analytic function on $\D.$ 
\begin{enumerate} 
\item[{\rm (i)}]	 The  operator $C_{(\psi,g)}: A^p_\omega\rightarrow A^q_\omega$ is bounded, then  
\begin{equation}\label{eq-c11}
\sup_{z\in \mathbb{D}}|g(\psi(z))|\,|\psi^\prime(z)|\,\frac{\tau(z)^{2/q}}{\tau(\psi(z))^{2/p}}\frac{(1+\varphi'(\psi(z))}{(1+\varphi'(z))}\,\frac{\omega(z)^{1/2}}{\omega(\psi(z))^{1/2}}\in L^\infty(\mathbb{D},dA).
\end{equation}

\item[{\rm (ii)}]The  operator $C_{(\psi,g)}: A^p_\omega\rightarrow A^q_\omega$ is compact, then
\begin{equation}\label{eq-n2}
\lim_{|\psi(z)|\to1^-} 
|g(\psi(z))|\,|\psi^\prime(z)|\frac{\tau(z)^{2/q}}{\tau(\psi(z))^{2/p}}\frac{(1+\varphi'(\psi(z))}{(1+\varphi'(z))}\,\frac{\omega(z)^{1/2}}{\omega(\psi(z))^{1/2}}=0.
\end{equation}
\end{enumerate}
\end{proposition}

\begin{proof}

We start with proving $(i)$. Suppose that the operator $C_{(\psi,g)}$ is bounded and we prove that  $\eqref{eq-c11}$ holds. Taking $\xi\in\D$ such that $|\psi(\xi)|>\rho_0\,$ we consider $ F_{\psi(\xi),n}$ is the test function defined in Lemma \ref{Borichevlemma}. By \eqref{Eq-gamma}, we have 
\begin{equation*}
\begin{split}
\|F_{\psi(\xi),n,p}\|^q_{A^p_{\om}} &\gtrsim\|C_{(\psi,g)}F_{\psi(\xi),n,p}\|^q_{A^q_{\om}}=\int_{\D}\frac{|F'_{\psi(\xi),n,p}(\psi(z))|^q}{(1+\varphi'(z))^q}\,|g(\psi((z))|^q |\psi^\prime (z)|^q\,\omega(z)^{\frac{q}{2}}\, dA(z)\\
&\gtrsim \tau(\xi)^2\frac{|F'_{\psi(\xi),n,p}(\psi(\xi))|^q}{(1+\varphi'(\xi))^q}\,|g(\psi(\xi))|^q |\psi^\prime (\xi)|^q\,\omega(\xi)^{\frac{q}{2}}.
\end{split}
\end{equation*}
Using Lemma \ref{BL3}, we get
\begin{equation*}
\begin{split}
\|F_{\psi(\xi),n,p}\|^q_{A^p_{\om}} 
&\gtrsim \tau(\xi)^2\,|g(\psi(\xi))|^q \,|\psi^\prime (\xi)|^q\,\frac{(1+\varphi'(\psi(\xi)))^q}{(1+\varphi'(\xi))^q}\,\frac{\omega(\xi)^{\frac{q}{2}}}{\omega(\psi(\xi)^{\frac{q}{2}}}.
\end{split}
\end{equation*}
 By Lemma \ref{Borichevlemma}, we obtain 
\begin{equation}\label{eqp11}
\begin{split}
1\gtrsim|g(\psi(\xi))|\,|\psi^\prime (\xi)|^q\, \frac{\tau(\xi)^{2/q}}{\tau(\psi(\xi))^{2/p}}\,\frac{(1+\varphi'(\psi(\xi)))}{(1+\varphi'(\xi))}\,\frac{\omega(\xi)^{\frac{1}{2}}}{\omega(\psi(\xi))^{\frac{1}{2}}}.
\end{split}
\end{equation}
On the other hand, for $|\psi(\xi)|\leq \rho_0$, we have 
$$
\sup_{\psi(\xi)\le\rho_0 }|g(\psi(\xi))|\,|\psi^\prime (\xi)|\, \frac{\tau(\xi)^{2/q}}{\tau(\psi(\xi))^{2/p}}\,\frac{(1+\varphi'(\psi(\xi)))}{(1+\varphi'(\xi))}\,\frac{\omega(\xi)^{\frac{1}{2}}}{\omega(\psi(\xi))^{\frac{1}{2}}}<\infty.
$$
This together with \eqref{eqp11}, completes the proof of $(i)$.

Now, we prove $(ii)$. Suppose that the operator $C_{(\psi,g)}$ is compact and we prove that  $\eqref{eq-n2}$ holds. Taking $\xi\in\D$ such that $|\psi(\xi)|>\rho_0$ and  the  bounded sequence 
$$\Big\{ f_{\psi(\xi),n,p}:=\frac{F_{\psi(\xi),n,p}}{\tau(\psi(\xi))^{2/p}},\quad\text{for}\quad |\psi(\xi)|>\rho_0\Big\} $$
of $A^p_\omega$ that converges uniformly to zero on compact subsets of $\D$ as $|\psi(\xi)|\to 1.$ By\eqref{Eq-gamma} and Lemma \ref{BL3}, we get 

\begin{equation*}
\begin{split}
\|C_{(\psi,g)}\,f_{\psi(\xi),n,p}\|^q_{A^q_{\om}}
&=\int_{\D}\frac{|f'_{\psi(\xi),n,p}(\psi(z))|^q}{(1+\varphi'(z))^q}\,|g(\psi(z))|^q |\psi^\prime (z)|^q\,\omega(z)^{\frac{q}{2}}\, dA(z)\\
&\gtrsim \tau(\xi)^2\,\frac{|f'_{\psi(\xi),n,p}(\psi(\xi))|^q}{(1+\varphi'(\xi))^q}\,|g(\psi(\xi))|^q |\psi^\prime (\xi)|^q\,\omega(\xi)^{\frac{q}{2}}\\
&\gtrsim|g(\psi(\xi))|^q\,|\psi^\prime (\xi)|^q \frac{\tau(\xi)^2}{\tau(\psi(\xi))^{2q/p}}\frac{(1+\varphi'(\psi(\xi)))^q}{(1+\varphi'(\xi))^q}\,\frac{\omega(\xi)^{\frac{q}{2}}}{\omega(\psi(\xi))^{\frac{q}{2}}}.
\end{split}
\end{equation*}
Since the compactness of the operator  $C_{(\psi,g)},$ the proof is complete.
\end{proof}

\begin{proposition}
  Let $ 0<p\le q< \infty.$ Suppose that  $\omega \in \mathcal{W},$  $\psi$ is an analytic self-map of $\D,$ and $g$ is an analytic function on $\D.$ 

\begin{enumerate} 
\item[{\rm (i)}]	 The  operator $C_g^\psi: A^p_\omega\rightarrow A^q_\omega$ is bounded, then  
\begin{equation}\label{eq-nc1}
\frac{\tau(z)^{2/q}}{\tau(\psi(z))^{2/p}}\frac{|g(\psi(z))||\psi^\prime(z)|}{(1+\varphi'(z))}\frac{\omega(z)^{1/2}}{\omega(\psi(z))^{1/2}}\in L^{\infty}(\D,dA).
\end{equation}
\item[{\rm (ii)}]The  operator $C_g^\psi: A^p_\omega\rightarrow A^q_\omega$ is compact, then
\begin{equation}\label{eq-nc2}
\lim_{|\psi(z)|\to1^-} \frac{\tau(z)^{2/q}}{\tau(\psi(z))^{2/p}}\frac{|g(\psi(z))||\psi^\prime(z)|}{(1+\varphi'(z))}\frac{\omega(z)^{1/2}}{\omega(\psi(z))^{1/2}}=0.
\end{equation}
\end{enumerate}  
\end{proposition} 

\begin{proof}
 Assume that $C_g^\psi: A^p_\omega\rightarrow A^q_\omega$ is bounded.  This equivalent to, by Theorem \ref{Theorem 1.1},  $M^\psi_{0,p,q}(g)\in L^{\infty}(\D,dA).$ Using \eqref{Eq-gamma} and \eqref{eqn:RK-Diag1}, we obtain
\begin{equation}\label{EsBV}
\begin{split}
M^\psi_{0,p,q}(g)(\psi(z))&=\int_{\D} |k_{p,\psi(z)}(\psi(\xi))|^q\,\frac{|g(\psi(\xi))|^q\,|\psi^\prime(\xi)|^q}{(1+\varphi'(\xi))^{q}}\,\,\omega(\xi)^{q/2}\, dA(\xi)\\
&\geq \int_{D_{\delta}(z)} |k_{p,\psi(z)}(\psi(\xi))|^q\,\frac{|g(\psi(\xi))|^q\,|\psi^\prime(\xi)|^q}{(1+\varphi'(\xi))^{q}}\,\,\omega(\xi)^{q/2}\, dA(\xi)\\
&\gtrsim \tau(z)^{2} \, |k_{p,\psi(z)}(\psi(z))|^q\,\frac{|g(\psi(z))|^q\,|\psi^\prime(z)|^q}{(1+\varphi'(z))^q}\,\,\omega(z)^{q/2}\\
&\gtrsim\frac{ \tau(z)^{2}}{\tau(\psi(z))^{2q/p}} \,\frac{|g(\psi(z))|^q\,|\psi^\prime(z)|^q}{(1+\varphi'(z))^q}\,\frac{\omega(z)^{q/2}}{\omega(\psi(z))^{q/2}},
\end{split}
\end{equation}
 which proves that \eqref{eq-nc1} holds.
 
Next if the operator $C_g^\psi: A^p_\omega\rightarrow A^q_\omega$ compact, then, by Theorem \ref{Theorem 1.1} and \eqref{EsBV}, we obtain
$$
    \lim_{|\psi(z)|\to1^-} \frac{\tau(z)^{2q}}{\tau(\psi(z))^{2q/p}}\frac{|g(\psi(z))|^q\,|\psi^\prime(z)|^q}{(1+\varphi'(z))^q}\frac{\omega(z)^{q/2}}{\omega(\psi(z))^{q/2}}=0,$$
which completes the proof.
\end{proof}

The following result is equivalent to another result similar to those given by Constantin and Pel\'{a}ez in the weighted Fock spaces \cite{Co-Pe}.

\begin{theorem}
Let $0<p, q\leq \infty.$ Suppose that $\omega \in \mathcal{W}$  and $g$ is an analytic function on $\D.$ If $\psi(z)=z$ for all $z\in \D,$ then 
\begin{enumerate}
\item[{\rm (a)}] For $p<q,$ the operator $C_{(\psi,g)}:A_\omega^p \rightarrow A_\omega^q$ is bounded if and only if $g=0.$
\item[{\rm (b)}] For $p>q,$ the operator $C_{(\psi,g)}:A_\omega^p \rightarrow A_\omega^q$ is compact if and only if $g \in L^r(\D, dA),$ where $r=pq/(p-q).$
\end{enumerate}
\end{theorem}
\begin{proof}
 We first prove $(a)$. Let $p<q$, and suppose that $C_{(id,g)}$ is bounded. Using Lemma \ref{lem:subHarmP} and \eqref{eqn:RK-Diag1}, we obtain
\begin{equation*}
\begin{split}
    |g(z)|^q &\asymp \tau(z)^{2q/p}\,|g(z)|^q \,|k_{p,z}(z)|^q\,\omega(z)^{q/2}\\
    &\lesssim \frac{\tau(z)^{2q/p}}{\tau(z)^2}\int_{D_\delta (z)} |g(s)|^q\,|k_{p,z}(s)|^q\,\omega(s)^{q/2}\, dA(s)\lesssim\frac{\tau(z)^{2q/p}}{\tau(z)^2}\,M_{1,p,q}^{id} (g)(z).
\end{split}
    \end{equation*}
    Furthermore, by boundedness of $C_{(id,g)},$ we have
    $$\sup_{z\in \D} |g(z)|^q \,\tau(z)^{2(1-{q/p})}\lesssim \sup_{z\in \D}M_{1,p,q}^{id} (g)(z)<\infty.$$
    Therefore, $g=0$, because $\tau(z)^{2(1-{q/p})}\rightarrow\infty$ as $|z|\rightarrow1.$
    
Now we prove (b). Using $\eqref{eqnorm}$, we have
    \begin{equation}\label{4.57}
        \|C_{(\psi,g)} f\|_{A_\omega^q}^q \asymp\int_\D \frac{|f^\prime(\psi(z)|^q\,|g(\psi(z)|^q\,|\psi^\prime(z)|^q}{(1+\varphi\prime(z))^q}\, \omega(z)^{q/2} \,dA(z) =\|f^\prime\|_{L^q(\mu_{\psi,\omega,g})}^q,
    \end{equation}
   by Lemma \ref{embed1} and Lemma \ref{embed2}, we have $C_{(\psi,g)}:A_\omega^p \rightarrow A_\omega^q$ is bounded if and only if $I_{\mu_{\psi,\omega,g}}:S_\omega^p \rightarrow L^q {(\mu_{\psi,\omega,g})}$ is bounded if and only if $I_{\mu_{\psi,\omega,g}}: S_{\omega^p}\rightarrow L^q(\mu_{\psi,\omega,g})$ is compact if and only if the function
    \begin{equation}\label{condation}
        F_\delta,_{\mu_{\psi,\omega,g}} (\varphi)(z):=\frac{1}{\tau(z)^2}\int_{D_\delta(z)} (1+\varphi^\prime(\xi))^q\,\omega(\xi)^{-q/2}\,d\mu_{\psi,\omega,g}(\xi)
    \end{equation}
    belongs to $L^{p/(p-q)}(\D,dA).$ Considering $\psi=id,$ we get
    $$d\mu_{\psi,\omega,g}(z)=\frac{|g(z)|^q}{(1+\varphi^\prime(z))^q}\,\omega(z)^{q/2}\,dA(z)$$
    and applying condition \eqref{condation}, we get
   \begin{equation*}
       \frac{1}{\tau(z)^2}\int_{D_\delta(z)}|g(\xi)|^q dA(\xi) \in L^{p/(p-q)}(\D,dA).
   \end{equation*}
   By Lemma \ref{lem:subHarmP}, we have that $g\in L^r(\D,dA),$ where $r=pq/(p-q).$

Assume next that $g\in L^r(\D,dA)$. By H\"older`s inequality and \eqref{eqnorm}, we get
   \begin{equation}
       \begin{split}
    \|C_{(id,g)}f\|_{A_\omega^q}^q &\asymp\int_\D \frac{|f^\prime(z)|^q\,|g(z)|^q}{(1+\varphi^\prime(z))^q} \,\omega(z)^{q/2}\,dA(z)\\
           &\lesssim\left(\int_\D \frac{|f^\prime(z)|^p\,\omega(z)^{p/2}}{(1+\varphi^\prime(z))^p}\,dA(z)\right)^{q/p}\left(\int_\D|g(z)|^r\,dA(z)\right)^{q/r}\asymp\|f\|_{A_\omega^p}^q\,\|g\|^q_{L^r(\D,dA)}\lesssim\|f\|_{A_\omega^p}^q,
       \end{split}
   \end{equation}
   which implies boundedness and completes the proof.
   \end{proof}

\begin{theorem} Let  $ 0<p\le q\le \infty.$ Suppose also that  $\omega \in \mathcal{W}$  and $g$ is an analytic function on $\D.$ 
\begin{enumerate} 
\item [\rm{(A)}] $M^{id}_{0,p,q}(g')\in L^{\infty}(\D,dA)$ if and only if 
\begin{equation}\label{eq-V7}
\frac{|g'(z)|}{(1+\varphi'(z))}\Delta\varphi(z)^{\frac{1}{p}-\frac{1}{q}}\in L^{\infty}(\D,dA).
\end{equation}

\item [\rm{(B)}] $\lim_{|z|\to1^-}M^{id}_{0,p,q}(g')=0$ if and only if 
\begin{equation}\label{eq-V8}
\lim_{|z|\to1^-}\frac{|g'(z)|}{(1+\varphi'(z))}\Delta\varphi(z)^{\frac{1}{p}-\frac{1}{q}}=0.
\end{equation}
\end{enumerate}
\end{theorem}
\begin{proof}
   We start with proving $(A).$ Assume that $M^{id}_{0,p,q}(g')\in L^{\infty}(\D,dA).$ It follows from \eqref{EsBV}, and changing $g$ by $g'$ and $\psi=id$, we get
\begin{equation}\label{EsBVn}
\begin{split}
M^{id}_{0,p,q}(g')(z)&=\int_{\D} |k_{p,z}(\xi)|^q\,\frac{|g'(\xi)|^q}{(1+\varphi'(\xi))^{q}}\,\,\omega(\xi)^{q/2}\, dA(\xi)\\
&\gtrsim\frac{ \tau(z)^{2}}{\tau(z)^{2q/p}} \,\frac{|g'(z)|^q}{(1+\varphi'(z))^q}\asymp\left(\frac{|g'(z)|}{(1+\varphi'(z))}\, \Delta\varphi(z)^{\frac{1}{p}-\frac{1}{q}}\right)^{q}.
\end{split}
\end{equation}
Thus, 
$$
    \frac{|g'(z)|}{(1+\varphi'(z))}\Delta\varphi(z)^{\frac{1}{p}-\frac{1}{q}}\in L^{\infty}(\D,dA).$$

For the reverse implication, assume that $$I(g,\varphi)(z):=\frac{|g'(z)|}{(1+\varphi'(z))}\, \Delta\varphi(z)^{\frac{1}{p}-\frac{1}{q}}\in L^{\infty}(\D,dA).$$ Using \eqref{eqn:Eq-NE1}, we get 
\begin{equation*}\label{EsBV1}
\begin{split}
M^{id}_{0,p,q}(g')(z)&=\int_{\D} |k_{p,z}(\xi)|^q\,\frac{|g'(\xi)|^q}{(1+\varphi'(\xi))^{q}}\,\,\omega(\xi)^{q/2}\, dA(\xi)\\
&\lesssim \left( \tau(z)^{2(1-q/p)} \int_{\D} |k_{q,z}(\xi)|^q\,\Delta\varphi(z)^{1-\frac{q}{p}}\,\omega(\xi)^{q/2}\, dA(\xi)\right)\sup_{z\in \D}\,(I(g,\varphi)(z))^q.\\
\end{split}
\end{equation*}
Since $\Delta\varphi(z)\asymp \tau(z)^{-2},$
\begin{equation}\label{EsBV2}
\begin{split}
M^{id}_{0,p,q}(g')(z)
&\lesssim \left( \int_{\D} |k_{q,z}(\xi)|^q\,\omega(\xi)^{q/2}\, dA(\xi)\right)\sup_{z\in \D}(I(g,\varphi)(z))^q\\
&=\|k_{q,z}\|^q_{A^q_{\om}}\sup_{z\in \D}(I(g,\varphi)(z))^q=\sup_{z\in \D}(I(g,\varphi)(z))^q.
\end{split}
\end{equation}
This completes the proof of $(A).$ 

The proof of (B) follows from Theorem \ref{Theorem 1.1}, \eqref{EsBVn}, and \eqref{EsBV2}.
\end{proof}

Before stating the next theorem, following Siskakis \cite{sis}, for a given weight $\omega$, we define the distortion function of $\omega$ by 
$$
\psi_\omega(r):=\frac{1}{\omega(r)}\int_{r}^{1}\omega(u)\, du, \, \, 0\le r< 1.
$$
According to (c) of Lemma 32 in \cite{Co-Pe},
 \begin{equation}\label{eq-V10}
\psi_\omega(r) \asymp (1+\varphi'(r))^{-1}, \,\,\textrm{for} \,\, r \in [0,1).
\end{equation}

\begin{theorem}\label{CAC1} Let  $ 0<q< p< \infty.$ Suppose also that  $\om \in \mathcal{W}$  and $g$ is an analytic function on $\D.$  The following statements are equivalent:
\begin{enumerate} 
\item[{\rm (a)}] The general transform function $$M^{id}_{0,p,q}(g') \in L^{\frac{p}{p-q}}(\D,d\lambda).$$

\item[{\rm (b)}] The function \begin{equation}\label{eq-V9}
\frac{|g'(z)|}{(1+\varphi'(z))}\in L^{\frac{pq}{p-q}}(\D,dA).
\end{equation}
\end{enumerate}
\end{theorem}

\begin{proof}
Let $0<q<p<\infty.$ Our first step is to verify that $(a)$ implies $(b).$ Suppose that $M_{0,p,q}^{id} (g^\prime) \in L^{\frac{p}{p-q}}(\D, d\lambda).$ Then, by \eqref{Eq-gamma}, we get 
$$M^{id}_{0,p,q}(g')(z)=\int_{\D} |k_{p,z}(\xi)|^q\,\frac{|g'(\xi)|^q}{(1+\varphi'(\xi))^{q}}\,\omega(\xi)^{q/2}\, dA(\xi)
\gtrsim \tau(z)^2 \,|k_{p,z}(z)|^q \,\frac{|g^\prime(z)|^q}{(1+\varphi^\prime(z))^q}\,\omega(z)^{q/2}.$$
Using Lemma \ref{lem:RK-PE1}, we have
\begin{equation*}
\begin{split}
M^{id}_{0,p,q}(g')(z)\gtrsim\frac{ \tau(z)^{2}}{\tau(z)^{2q/p}} \,\frac{|g'(z)|^q}{(1+\varphi'(z))^{q}}\asymp\left(\frac{|g'(z)|}{(1+\varphi'(z))}\, \Delta\varphi(z)^{\frac{1}{p}-\frac{1}{q}}\right)^{q}.
\end{split}
\end{equation*}
In this case, we conclude that 
\begin{equation*}\label{EsBVn1}
\begin{split}
\tau(z)^{2(\frac{q}{p}-1)}M^{id}_{0,p,q}(g')(z)\gtrsim\left(\frac{|g'(z)|}{(1+\varphi'(z))}\right)^{q}.
\end{split}
\end{equation*}
By our assumption and the fact that $\tau(z)^{2q/p}$ is bounded, (b) is true.

Conversely, put  $r=\frac{pq}{(p-q)},$ by  H\"older's inequality and , we obtain
\begin{equation*}\label{EsBV3}
\begin{split}
&M^{id}_{0,p,q}(g')(z)^{p/(p-q)}=\left(\int_{\D} |k_{p,z}(\xi)|^q\,\frac{|g'(\xi)|^q}{(1+\varphi'(\xi))^{q}}\,\,\omega(\xi)^{q/2}\, dA(\xi)\right)^{p/(p-q)}\\
&\le\|K_z\|^{-r}_{A^p_{\om}} \left(\int_{\D} |K_{z}(\xi)|^{\frac{r}{2}} \left(\frac{|g'(\xi)|}{1+\varphi'(\xi)}\right)^{r} \omega(\xi)^{\frac{r}{4}} \, dA(\xi)\right)
\cdot\left(\int_{\D} |K_{z}(\xi)|^{\frac{p}{2}}  \omega(\xi)^{\frac{p}{4}}\, dA(\xi)\right)^{\frac{q}{(p-q)}}\\
&=\frac{\|K_z\|^{r/2}_{A^{p/2}_{\om}}}{\|K_z\|^{r}_{A^p_{\om}}} \int_{\D} |K_{z}(\xi)|^{\frac{r}{2}} \left(\frac{|g'(\xi)|}{1+\varphi'(\xi)}\right)^{r} \omega(\xi)^{\frac{r}{4}} \, dA(\xi).
\end{split}
\end{equation*}
Using Theorem \ref{RK-PE} and Fubini's theorem, we have 
\begin{equation*}\label{EsBV4}
\begin{split}
&\int_{\D} M^{id}_{0,p,q}(g')(z)^{p/(p-q)}\frac{dA(z)}{\tau(z)^{2}}\\
&\lesssim\int_{\D}\left(\frac{|g'(\xi)|}{1+\varphi'(\xi)}\right)^{r} \omega(\xi)^{\frac{r}{4}}  \left( \int_{\D} |K_{\xi}(z)|^{\frac{r}{2}} \, \omega(z)^{\frac{r}{4}}\,\tau(z)^{r-2} \,dA(z) \right)dA(\xi).
\end{split}
\end{equation*}
Since $$\omega(\xi)^{\frac{r}{4}}  \left( \int_{\D} |K_{\xi}(z)|^{\frac{r}{2}} \, \omega(z)^{\frac{r}{4}}\,\tau(z)^{r-2} \,dA(z) \right)\lesssim 1$$ (see Lemma \ref{nEstim}), the proof is complete.
\end{proof}

In the following theorem, we prove that our necessary and sufficient conditions of boundedness and compactness of classical Volterra operators are equivalent to those results given by Pau and Pel\`{a}ez in \cite{PP1}.

\begin{theorem}\label{Jor-PelC} Let  $ 0<p,q< \infty.$ Suppose also that  $\omega \in \mathcal{W}$  and $g$ is an analytic function on $\D.$  
\begin{enumerate} 

\item[(I)] For  $p=q,$  we have the following statements
\begin{enumerate} 
 \item[(a)]  $M^{id}_{0,p,q}(g^\prime) \in L^{\infty}(\D,dA)$ \quad if and only if \,$\psi_\omega(z)\,|g^\prime(z)|\in L^{\infty}(\D,dA).$
\item[(b)]  $\lim_{|z|\to1} M^{id}_{0,p,q}(g^\prime)=0$ if and only if 
$\,\lim_{|z|\to1}\psi_\omega(z)\,|g^\prime(z)|=0.$
\end{enumerate}
\item[(II)] For  $p<q,$ with 
\begin{equation}\label{jcn}
\Delta\varphi(z)\asymp ((1-|z|)^{t}\,\psi_\omega(z))^{-1}, \,\, z\in \D, \,\, \textrm{for some}\,\,\, t\geq 1
\end{equation}
the following statements are equivalent:
\begin{enumerate} 
\item[(c)]  $M^{id}_{0,p,q}(g^\prime) \in L^{\infty}(\D,dA).$
\item[(d)]  The function $g$ is constant.
\end{enumerate}
\end{enumerate}
\end{theorem}

\begin{proof}
First, we begin with the proof of $(a)$ of $(I)$ when $p=q.$ Assume that $\psi_\om(z)\,|g'(z)|\in L^{\infty}(\D,dA)$  and we prove that $M^{id}_{0,p,p}(g') \in L^{\infty}(\D,dA)$. By using \eqref{eq-V10}, we have 
\begin{equation}\label{CJ}
\begin{split}
M^{id}_{0,p,p}(g')(z)&=\int_{\D} |k_{p,z}(\xi)|^p\,\frac{|g'(\xi)|^p}{(1+\varphi'(\xi))^{p}}\,\,\omega(\xi)^{p/2}\, dA(\xi)\\
&\asymp\sup_{\xi\in \D}\left(\psi_\om(\xi)\,|g'(\xi)|\right)^p \,\left(  \int_{\D} |k_{p,z}(\xi)|^p\,\omega(\xi)^{p/2}\, dA(\xi)\right)\\
&=\sup_{\xi\in \D}\left(\psi_\om(\xi)\,|g'(\xi)|\right)^p\, \|k_{p,z}\|^p_{A^p_\om}=\sup_{\xi\in \D}\left(\psi_\om(\xi)\,|g'(\xi)|\right)^p.
\end{split}
\end{equation}
Thus, $M^{id}_{0,p,q}(g') \in L^{\infty}(\D,dA)$ if and only if $\psi_\om(z)\,|g'(z)|\in L^{\infty}(\D,dA).$
The proof of (b) follows easily from~\eqref{CJ}.

Next, we prove (II). It is easy to see that (d) implies (c). Note that the  weighted Bergman space $A^p(\omega),$ defined in \cite{PP1}, is  the same as  the Bergman spaces $A^p_W, $ with $W=\omega^{2/p}$ and for $0<p<\infty.$ 
Moreover, 
$$
M^{id}_{0,p,q}(g')(z) = \int_{D_{\delta}(z)}|k_{p,z}(\xi)|^q\,\frac{|g'(\xi)|^q}{(1+\varphi'(\xi))^{q}}\,\omega(\xi) \,dA(\xi),
$$
and  \eqref{eqn:RK-Diag1} is transformed to 
\begin{equation}\label{eqn:RK-Diag1n}
|k_ {p,z}(\zeta)|^q \, \omega(\zeta)^{q/p} \asymp \tau(z)^{-2q/p} ,\qquad \zeta \in D_\delta(z),
\end{equation}
where $k_{p,z}(\xi)=K_z(\xi)/\|K_{p,z}\|_{A^p(\om)}.$

 For $p<q$ and write $s=\frac{1}{p}-\frac{1}{q},$ using   \eqref{eq-V10} and  successively \eqref{eqn:asymptau}, \eqref{Eq-gamma} with $(\beta = 1-\frac{q}{p})$ and \eqref{eqn:RK-Diag1n} we give
\begin{equation*}\label{CJ2}
\begin{split}
\,\left(\|K_z\|^{2s}_{A^2(\om)}\psi_\om(z)|g'(z)|\right)^q&\lesssim\frac{\|K_z\|^{2qs}_{A^2(\om)}}{\tau(z)^{2}\omega(z)^{1-\frac{q}{p}}}\int_{D_{\delta}(z)}\frac{|g'(\xi)|^q}{(1+\varphi'(\xi))^{q}}\,\omega(\xi)^{1-\frac{q}{p}} \,dA(\xi)\\
&\lesssim\frac{1}{\tau(z)^{2q/p}}\int_{D_{\delta}(z)}\frac{|g'(\xi)|^q}{(1+\varphi'(\xi))^{q}}\,\omega(\xi)^{1-\frac{q}{p}} \,dA(\xi)\\
&\lesssim\int_{D_{\delta}(z)}|k_{p,z}(\xi)|^q\frac{|g'(\xi)|^q}{(1+\varphi'(\xi))^{q}}\,\omega(\xi)\, dA(\xi)\\
&\lesssim M^{id}_{0,p,q}(g')(z)< \infty.
\end{split}
\end{equation*}
Thus,  in order  to prove that  the function $g'$ is vanish on $\D,$ it is enough to see that $\|K_z\|^{2s}_{A^2(\om)}\psi_\om(z)$ goes to infinity as $|z|\to 1.$ Using \eqref{Eq-NE} and \eqref{jcn}, we get 
$$
\|K_z\|^{2s}_{A^2(\om)}\,\psi_\om(z)\asymp \frac{\tau(z)^{2(1-s)}}{(1-|z|)^{t}\,\omega(z)^{s}}.
$$
It remains to note that $\lim_{\substack{|z|\to 1}} \|K_z\|^{2s}_{A^2(\om)}\,\psi_\om(z)=\infty$ (see Lemma 2.3 in \cite{PP1}). 
\end{proof}

\section{Schatten class membership}\label{Section 4}
\begin{proof}[Proof of Theorem 1.2]
Notice first that, for any $f,h \in A_{\omega}^{2},$
\begin{equation}\label{define inner of CC*}
    \begin{split}
        \langle\left(C_{g}^{\psi}\right)^{\ast}(C_{g}^{\psi})f, h\rangle_\omega &= \langle C_{g}^{\psi} f,C_{g}^{\psi}h\rangle_\omega=\int_{\mathbb{D}}\left(C_{g}^{\psi}f(z)\right)^{\prime} \overline{\left(C_{g}^{\psi} h(z)\right)^\prime}\,\,\frac{\omega(z)}{(1+\varphi^\prime(z))^2}\,dA(z)\\
        &=\int_{\mathbb{D}}f(\psi(z))g(\psi(z))\psi^\prime(z)\,\overline{h(\psi(z))g(\psi(z))\psi^\prime(z)}\,\,\frac{\omega(z)}{(1+\varphi^\prime(z))^2}\,dA(z)\\
&=\int_{\mathbb{D}}f(z)\,\overline{h(z)}\,d\mu_{\psi,2}(z)=\int_{\mathbb{D}}f(z)\,\overline{h(z)}\,\,\omega(z)\,d\nu_{\psi,2}(z),
    \end{split}
\end{equation}
where $d\mu_{\psi,2}=\omega(z)\,\,dm_{\psi,2}(z).$ Denote by $T_{m_{\psi,2}}$ the Toeplitz operator with a positive measure $m_{\psi,2}$ defined by
$$
T_{m_{\psi,2}}=\int_\mathbb{D}\,f(\xi)\,\overline{K_z(\xi)}\,\omega(\xi)\,dm_{\psi,2}(\xi), \qquad \text{for}\,\,f\in A_{\omega}^2.
$$
Applying Fubini’s theorem and the reproducing kernel formula, we have
\begin{equation}\label{ eq T }
    \begin{split}
        \langle T_{m_{\psi,2}}\,f, h\rangle_\omega &= \int_{\mathbb{D}}\left(\int_{\mathbb{D}}f(\xi)\,K_{\xi}(z)\,\omega(\xi)\,dm_{\psi,2}(\xi)\right)\,\overline{h(z)}\, \omega(z) \,dA(z)\\
        &=\int_{\mathbb{D}}f(\xi)\,\overline{\langle h,K_{\xi}\rangle_{\omega}}\,\omega(\xi)\, dm_{\psi,2}(\xi)=\int_{\mathbb{D}}f(\xi)\,\overline{h(\xi)} \,\omega(\xi)\,dm_{\psi,2}(\xi),
    \end{split}
\end{equation}
for any $f, h \in A_{\omega}^2.$ Combining this with \eqref{define inner of CC*}, we get
$$\langle \left(C_{g}^{\psi}\right)^\ast\,\left(C_{g}^{\psi}\right)f, h\rangle_\omega = \langle \,T_{m_{\psi,2}}\,f,h\rangle_\omega, \qquad \text{ for every} \,\, f,h \in A_{\omega}^2.$$
Thus, $\left(C_{g}^{\psi}\right)^\ast\,\left(C_{g}^{\psi}\right)= \,T_{m_{\psi,2}}.$
Hence,  $C_{g}^{\psi}$ belongs to $\mathcal{S}_p(A_\omega ^2)$ if and only if $\,T_{m_{\psi,2}} $ is in $\mathcal{S}_{p/2}(A_{\omega}^2),$ which is equivalent to $\widehat{m_{\psi,2}} \in L^{p/2}(\mathbb{D}, d\lambda), $ by Theorem 4.6 in \cite{APJ}. This is also equivalent to that $G_{t}(m_{\psi,2})$  belongs to $L^{p/2}(\mathbb{D},d\lambda)$ for $t>0,$ by Lemma 7.1 in \cite{Hi-w}. Since
\begin{align*}\label{eq G0,2}
   G_{2}(m_{\psi,2})(z)&= \int_{\mathbb{D}} |k_{2,z}(\xi)|^2 \,\omega(\xi)\, dm_{\psi,2}(\xi)=\int_{\mathbb{D}} |k_{2,z}(\xi)|^2 \omega(\xi)\, \omega(\xi)^{-1}\, d\mu_{\psi,2}(\xi)\\
   &=\int_{\mathbb{D}} |k_{2,z}(\xi)|^2 \,d\mu_{\psi,2}(\xi)=\int_{\mathbb{D}} |k_{2,z}(\psi(\xi))|^2 \,|g(\psi(\xi))|^2\,|\psi^\prime(\xi)|^2\,\frac{\omega(\xi)}{(1+\varphi^\prime(\xi))^2}\,dA(\xi) = M_{0,2,2}^{\psi}(g)(z)
\end{align*}
for $z\in\D$, which completes the first statement of Theorem \ref{Theorem 1.2}. 

To prove that $GV (\psi,g )$ belongs to the Schatten $p$-class $S_p(A_\omega^2)$, it suffices to follow the same arguments used in the preceding part of the proof. We omit the details and leave them to the interested reader.
\end{proof}

In the following proposition, when $\psi=id$, we prove that our necessary and sufficient condition for Schatten class membership is equivalent to Theorem 3 of \cite{PP1} given by Pau and Pel\`{a}ez.

\begin{proposition}\label{Jor-PelC1} let  $\om \in \mathcal{W}$  and $g$ is an analytic function on $\D.$  
\begin{enumerate} 
\item[(I)]  Let  $ 1<p<\infty$, the following conditions are equivalent:
 \begin{enumerate} 
\item[(a)]  $M^{id}_{0,2,2}(g') \in L^{p/2}(\D,d\lambda).$
\item[(b)] $\psi_\om(z)|g'(z)|\in L^{p}(\D,d\lambda),$ 
\end{enumerate}
where $d\lambda(z)= dA(z)/\tau^2(z).$
\item[(II)]
 For  $0< p\leq 1,$  the following statements are equivalent:
\begin{enumerate} 
\item[(c)]  $M^{id}_{0,2,2}(g') \in L^{p/2}(\D,d\lambda).$
\item[(d)]  The function $g$ is constant.
\end{enumerate}
\end{enumerate}
\end{proposition}
\begin{proof}
First, we prove that $(a)$ implies $(b)$,that is,  assume that $\psi_\omega(z)\,|g^\prime(z)|\in L^{p}(\mathbb{D},d\lambda)$  and we need to prove that $M^{id}_{0,p,p}(g^\prime) \in L^{p/2}(\mathbb{D},d\lambda)$. By using \eqref{eq-V10}, we obtain
\begin{equation}\label{CJ3}
\begin{split}
M^{id}_{0,2,2}(g')(z)&=\int_{\D} |k_{p,z}(\xi)|^2\,\frac{|g'(\xi)|^2}{(1+\varphi'(\xi))^{2}}\,\,\omega(\xi)\, dA(\xi)\\
&\lesssim\sup_{\xi\in \D}\left(\psi_\om(\xi)\,|g'(\xi)|\right)^2 \,\left(  \int_{\D} |k_{p,z}(\xi)|^2\,\omega(\xi)\, dA(\xi)\right)\\
&=\sup_{\xi\in \D}\left(\psi_\om(\xi)\,|g'(\xi)|\right)^2\, \|k_{p,z}\|^2_{A^2_\om}=\sup_{\xi\in \D}\left(\psi_\om(\xi)\,|g'(\xi)|\right)^2.
\end{split}
\end{equation}

Conversely, suppose that  $M^{id}_{0,2,2}(g')$ belongs to  $L^{p/2}(\D,d\lambda)$, then by using again  \eqref{eq-V10} and  respectively \eqref{Eq-gamma}, \eqref{eqn:asymptau} and  \eqref{eqn:RK-Diag1}, we get 
\begin{align*}
\left(\psi_\om(z)\,|g'(z)|\right)^2&\asymp\frac{|g'(z)|^2}{(1+\varphi'(z))^{2}}\lesssim\frac{1}{\tau(z)^2}\int_{D_{\delta}(z)}\frac{|g'(\xi)|^2}{(1+\varphi'(\xi))^{2}}\,dA(\xi)
\lesssim\int_{D_{\delta}(z)}|k_{p,z}(\xi)|^2\frac{|g'(\xi)|^2}{(1+\varphi'(\xi))^{2}}\,\omega(\xi)\, dA(\xi)\\
&\lesssim\int_{\D}|k_{p,z}(\xi)|^2\frac{|g'(\xi)|^2}{(1+\varphi'(\xi))^{2}}\,\omega(\xi)\, dA(\xi)= M^{id}_{0,2,2}(g')(z),
\end{align*}
This together with \eqref{CJ3} shows that $M^{id}_{0,2,2}(g') \in L^{p/2}(\D,d\lambda)$ if and only if $\psi_\om(z)\,|g'(z)|\in L^{p}(\D,d\lambda)$,
which finishes the proof of (I).

Next, we prove (II). It is clear that $(d)$ implies $(c).$ Now, we assume that $(c)$ is true and we prove that the function $g$ is  constant. Suppose that  $M^{id}_{0,2,2}(g') \in L^{p/2}(\D,d\lambda).$  The application of  \eqref{def tau},  \eqref{jcn} and  the fact that $\Delta\varphi(z)\asymp \tau(z)^{-2}$   imply that 
\begin{equation*}
\begin{split}
    \int_\mathbb{D} \frac{|g^\prime(z)|^p}{(1-|z|)^{tp}\,(1-|z|)^{2(1-p)}}\, dA(z)&\lesssim \int_\mathbb{D} \frac{|g^\prime(z)|^p}{(1-|z|)^{tp}\,\tau(z)^{2(1-p)}}\,dA(z)\asymp \int_\mathbb{D}\frac{|g^\prime(z)|^p \,\Delta\varphi(z)^{1-p}}{(1-|z|)^{tp}}\,dA(z)\\
&\asymp\int_\mathbb{D}\frac{|g^\prime(z)|^p \,\Delta\varphi(z)}{(1-|z|)^{tp}\,(1-|z|)^{-tp}\,\psi_\omega(z)^{-p}}\,dA(z)\\
&\asymp\int_\mathbb{D} |g^\prime(z)|^p\, \psi_\omega(z)^p \,\Delta\varphi(z)\,dA(z) \asymp \int_\mathbb{D}|g^\prime(z)|^p \,\psi_\omega(z)^p d\lambda(z).
\end{split}  
\end{equation*}
Using our assumption and the fact that $M^{id}_{0,2,2}(g') \in L^{p/2}(\D,d\lambda)$ is equivalent to $\psi_\om(z)|g'(z)|\in L^{p}(\D,d\lambda)$ for all $0<p<\infty$, we have 
\begin{equation*}
\begin{split}
    \int_\mathbb{D} \frac{|g^\prime(z)|^p}{(1-|z|)^{tp}\,(1-|z|)^{2(1-p)}}\, dA(z)
& \lesssim \int_\mathbb{D}|g^\prime(z)|^p \,\psi_\omega(z)^p d\lambda(z) \asymp \int_\mathbb{D}|M_{0,2,2}^{id}(g^\prime)(z)|^{p/2} d\lambda(z) <\infty.
\end{split}  
\end{equation*}
Therefore, it follows $(t-2)p+2\geq 1,$ and consequently $g^\prime \equiv 0,$ which completes the proof.
\end{proof}

\bibliographystyle{alpha}

\begin{thebibliography}{aa}
\bibitem{Hi-w} H. Arroussi, \emph{Weighted composition operators on Bergman spaces $A^p_\omega$}, Math. Nachr. (2020).

\bibitem{Hi-w2} H. Arroussi, \emph{Bergman spaces with exponential type weights}, J Inequal Appl, 193 (2021).
\bibitem{HGJ} H. Arroussi, H. Gissy, and J.A. Virtanen, Generalized Volterra type integral operators on large Bergman spaces, Bulletin des Sciences Mathématiques, (2023), 182: 103226.



\bibitem{APJ} H. Arroussi, I. Park and J. Pau, \emph{Schatten class Toeplitz operators acting on large weighted Bergman
spaces}, Studia Math. 229 (2015), no. 3, 203–221.
\bibitem{ASSER} S. Asserda and A. Hichame, \textit{Pointwise estimate for the Bergman kernel of the weighted Bergman spaces with exponential type weights,} C. R., Math., Acad. Sci. Paris 352 (2014), 13-16.

\bibitem{BDK} A. Borichev, R. Dhuez and K. Kellay, \textit{Sampling and interpolation in
large Bergman and Fock spaces,} J. Funct. Anal. 242 (2007), 563--606.
\bibitem{Co-Pe} O. Constantin and J. A. Pel\`{a}ez, \emph{Integral Operators, Embedding Theorems and a
Littlewood–Paley Formula on Weighted Fock Spaces,} J. Geom. Anal., 26(2015), 1109--1154.
\bibitem{Con&Pel} O. Constantin and J. A. Pel{\'a}ez, \emph{Integral operators, embedding theorems and a Littlewood--Paley formula on weighted Fock spaces}, The Journal of Geometric Analysis. Springer, 26 (2016), 1109-1154.
\bibitem{YLTD} Y. Deng,  L. Huang, T. Zhao, and D. Zheng, \emph{ Bergman projection and Bergman spaces}, Journal of Operator Theory 46, no. 1 (2001): 3–24.
\bibitem{D2} M. Dostanic, \emph{Integration operators on Bergman spaces with exponential weights}, Revista Mat. Iberoamericana 23 (2007), 421--436.
\bibitem {GP1}  P. Galanopoulos and  J. Pau, \emph{Hankel operators on large weighted Bergman spaces}, Ann. Acad. Sci. Fenn. Math. 37 (2012), 635--648.
\bibitem {Galano} P. Galanopoulos, \emph{Schatten class Hankel operators on Large Bergman Spaces,} arXiv preprint arXiv:(2021) 2106.05898.
\bibitem{HALMOS} P. R. Halmos, \emph{Measure Theory,} Springer-Verlag, New York, 1974.
\bibitem{Hakan1} H. Hedenmalm, \emph{ An off-diagonal estimate of Bergman kernels,} Journal de mathématiques pures et appliquées, 79.2 (2000), 163-172.
\bibitem{Hakan} H. Hedenmalm, B. Korenblum and K. Zhu,
  \emph{Theory of Bergman spaces} Springer 2012.
\bibitem{xiaof} Z. Hu, X. Lv and  A. Schuster, \emph{Bergman spaces with exponential weights}, J. Funct. Anal. 276 (2019), 1402-1429 .
\bibitem{KrMc} T. Kriete and B. MacCluer, \emph{Composition operators on Large Weighted Bergman spaces}, Indiana Univ. Math. J. 41 (1992),
755--788.

\bibitem{LR1} P. Lin and R. Rochberg, \emph{Hankel operators on the weighted Bergman spaces with exponential type weights}, Integral Equations Oper. Theory 21 (1995), 460--483.

\bibitem{LR2} P. Lin and R. Rochberg, \emph{Trace ideal criteria for Toeplitz and Hankel operators on the weighted Bergman spaces with exponential type weights}, Pacific J. Math. 173 (1996), 127--146.
\bibitem{Lue2} D. Luecking, \emph{ A technique for characterizing Carleson measures on bergman spaces}, Proc. Amer. Math. Soc. 87 (1983), 656--660.
\bibitem{MMO} N. Marco, X. Massaneda, J. Ortega-Cerda,\emph{  Interpolating and sampling sequences for entire functions,
} Geom. Funct. Anal. 13 (2003), 862-914.
\bibitem{Tesfa} T. Mengestie,\emph{ Product of Volterra Type Integral and Composition Operators on Weighted Fock Spaces}. J Geom Anal 24 (2014), 740–755. \bibitem{Tesfa0} T. Mengestie, \emph {Volterra type and weighted composition operators on weighted Fock spaces}. Integral Equations and Operator Theory, (2013), 76.1: 81-94.
\bibitem{Tesfa00} T. Mengestie, \emph{Generalized Volterra companion operators on Fock spaces}. Potential Analysis, (2016), 44: 579-599.
\bibitem{Tesfa1} T. Mengestie, \emph{Schatten-class generalized Volterra companion integral operators}. Banach J. Math. Anal. (2016): 267-280.
\bibitem{Tesfa2}T. Mengestie and M. H. Takele, \emph{Integral and Weighted Composition Operators on Fock-type Spaces}, Bulletin of the Malaysian Mathematical Sciences Society. Springer,  46.2 (2023): 80.
\bibitem{O} V.~L. Oleinik, \emph{Embedding theorems for weighted classes of harmonic and analytic functions}, J. Soviet. Math. 9 (1978), 228--243.
\bibitem{Iny-Park2} I . Park,  \emph{Compact differences of composition operators on large weighted Bergman spaces}, J. Math. Anal. Appl. 479 (2019), no. 2, 1715–1737.
\bibitem{Iny-Park} I. Park, \emph{The weighted composition operators on the large weighted Bergman spaces}, J. Math. Anal. Appl. 479 (2019), 1715--1737 .
\bibitem{PP1}  J. Pau and J.  A. Pel\'{a}ez, \emph{Embedding theorems and integration operators on Bergman spaces with rapidly decreasing weights}, J. Funct. Anal. 259 (2010), 2727--2756.
\bibitem{PP2} J. Pau and J.  A.  Pel\'{a}ez, \emph{Volterra type operators on Bergman spaces with exponential weights}, Contemp. Math. 561 (2012), 239--252.
\bibitem{JJK} J. {\'A}. Pel{\'a}ez, J. R{\"a}tty{\"a},  and K. Sierra,  \emph{Berezin transform and Toeplitz operators on Bergman spaces induced by regular weights}, The Journal of Geometric Analysis, Springer. 28(2018), 656-687.
\bibitem{sis} A. Siskakis, \emph{Weighted integrals and conjugate functions in the unit disk}, Acta Sci. Math. (Szeged) 66 (2000), 651–
664.
\bibitem{smith} M. P. Smith, \emph{Testing Schatten class Hankel operators and Carleson embeddings via reproducing kernels}, Journal of the London Mathematical Society 71.1 (2005), 172-186.

\bibitem{tija} M. Tijani, \emph{Compact composition operators on Besov spaces }, Trans. Amer. Math. Soc. 355 (2003), 4683--4698.



\bibitem{Zhu} K. Zhu, \emph{Operator theory in function spaces}, 2nd ed.,
Amer. Math.
Soc., Providence, 2007.
\end{thebibliography}

\end{document}